\documentclass[preprint,12pt]{elsarticle}

\usepackage{lineno,hyperref}
\modulolinenumbers[5]

\usepackage{ifpdf}
\usepackage{bm}
\usepackage{mathdots}
\usepackage{soul, color}
\usepackage{amsmath}
\usepackage{mathrsfs}
\usepackage{subfig}
\usepackage{graphicx}
\usepackage{epstopdf}
\usepackage{epsfig}
\usepackage{yfonts}

\newtheorem{theorem}{Theorem}[section]
\newtheorem{lemma}[theorem]{Lemma}

\newtheorem{corollary}[theorem]{Corollary}

\newtheorem{definition}{Definition}[section]

\newenvironment{proof}[1][Proof]{\begin{trivlist}
		\item[\hskip \labelsep {\bfseries #1}]}{\end{trivlist}}
\newtheorem{example}{Example}
\newenvironment{remark}[1][Remark]{\begin{trivlist}
		\item[\hskip \labelsep {\bfseries #1}]}{\end{trivlist}}

\usepackage{amssymb}
\journal{Journal}
\begin{document}
	
	\begin{frontmatter}
		
		

		\title{The eigenvalue distribution of special $2$-by-$2$ block  matrix sequences, with applications to the case of symmetrized Toeplitz structures}
		
		\author{Paola Ferrari}% \fnref{label1}}
		\address{Department of Science and high Technology, University of Insubria, Via Valleggio 11, Como, 22100,
			Italy}
		\ead{pferrari@uninsubria.it}
		\author{Isabella Furci}% \fnref{label1}}
		\address{Department of Science and high Technology, University of Insubria, Via Valleggio 11, Como, 22100,
			Italy}
		\ead{ifurci@uninsubria.it}
		\author{Sean Hon\corref{cor1}}% \fnref{label1}}
		\address{Mathematical Institute, University of Oxford, Radcliffe Observatory Quarter, Oxford, OX2 6GG,
			United Kingdom}
		\ead{hon@maths.ox.ac.uk}
		
		 \cortext[cor1]{Corresponding author}
		
		%%%%%
		
		\author{Mohammad Ayman Mursaleen\corref{cor2}}% \fnref{label1}}
		\address{Department of Science and high Technology, University of Insubria, Via Valleggio 11, Como, 22100,
			Italy}
		\ead{mamursaleen@uninsubria.it}
		
		%%%%%%%%%%%%%
		
		\author{Stefano Serra-Capizzano\corref{cor2}}% \fnref{label1}}
		\address{Department of Science and high Technology, University of Insubria, Via Valleggio 11, Como, 22100,
			Italy}
		\ead{stefano.serrac@uninsubria.it}
		%%\cortext[cor1]{Corresponding author}

		\begin{abstract}
			Given a Lebesgue integrable function $f$ over $[0,2\pi]$, we consider the sequence of matrices $\{Y_nT_n[f]\}_n$, where $T_n[f]$ is the $n$-by-$n$ Toeplitz matrix generated by $f$ and $Y_n$ is the flip permutation matrix, also called the anti-identity matrix. Because of the unitary character of $Y_n$, the singular values of $T_n[f]$ and $Y_n T_n[f]$ coincide. However, the eigenvalues are affected substantially by the action of the matrix $Y_n$. Under the assumption that the Fourier coefficients are real, we prove that  $\{Y_nT_n[f]\}_n$ is distributed in the eigenvalue sense as
			\[
			\phi_g(\theta)=\left\{
			\begin{array}{cc}
			g(\theta), & \theta\in [0,2\pi], \\
			-g(-\theta), & \theta\in [-2\pi,0),  
			\end{array}
			\right.\,
			\]
			with $g(\theta)=|f(\theta)|$. We also consider the preconditioning introduced by Pestana and Wathen and, by using the same arguments, we prove that the preconditioned sequence is distributed in the eigenvalue sense as $\phi_1$, under the mild assumption that $f$ is sparsely vanishing. We emphasize that the mathematical tools introduced in this setting have a general character and in fact can be potentially used in different contexts.
			A number of numerical experiments are provided and critically discussed.
		\end{abstract}
		
		\begin{keyword}
			Toeplitz matrices \sep Hankel matrices \sep circulant preconditioners %\sep PCG \sep PMINRES
			
			\MSC[] 15B05 \sep 65F15 \sep 65F08
		\end{keyword}
		
	\end{frontmatter}

\section{Introduction}

Given a Lebesgue integrable function $f$ defined on $[-\pi,\pi]$, i.e. $f\in L^{1}([-\pi,\pi])$, and periodically extended to the whole real line, we consider the Toeplitz matrix $T_n[f]$ of size $n$ generated by $f$. For any $n$, the entries of $T_n[f]$ are defined via the Fourier coefficients $\{a_k(f)\}_k$,
	$a_k=a_k(f)$, $k\in \mathbb Z$, of $f$ in the sense that 
	\[
	\left[T_n[f]\right]_{s,t}=a_{s-t},\quad \quad s,t \in \{1,\ldots,n\}.
	\]
	In the case where the Fourier coefficients are real, namely the corresponding $T_n[f]$ is (real) nonsymmetric, Pestana and Wathen \cite{doi:10.1137/140974213} recently suggested that one can first premultiply $T_{n}[f]$ by the anti-identity matrix $Y_{n}\in \mathbb{R}^{n \times n}$ defined as
	\[
	Y_{n}=\begin{bmatrix}{}
	& & 1 \\
	& \iddots & \\
	1 &  & \end{bmatrix}
	\]
	in order to obtain the symmetrized matrix $Y_{n}T_n[f]$ (i.e. a Hankel matrix). They then introduced an absolute value circulant preconditioner $|C_n|$ and showed, under certain assumptions, that the eigenvalues of $|C_n|^{-1}Y_nT_n[f]$ are clustered around $\pm1$. The same techniques were also proven applicable to functions of Toeplitz matrices in \cite{Hon2018,Hon2018148,2018arXiv180710929H}.
	
	In this work, considering the symmetrized Toeplitz matrix sequences $\{Y_nT_n[f]\}_n$ with $T_n[f]$ generated by $f \in L^1([-\pi,\pi])$, we provide theorems that precisely describe its singular value and spectral distribution, which further extend our previous results in \cite{HMS}. It was shown in \cite{HMS} that roughly half of the eigenvalues of $Y_nT_n[f]$ are negative/positive, when the dimension is sufficiently large and $f$ is sparsely vanishing,  i.e. its set of zeros is of (Lebesgue) measure zero.
	
	We first give a general distributional result, that is Theorem \ref{thm:main_hon-gen}, regarding the eigenvalues of special $2$-by-$2$  block matrix sequences and then furnish the distribution analysis of $\{Y_nT_n[f]\}_n$ in the sense of eigenvalues, under the only assumption that $f$ is Lebesgue integrable with real Fourier coefficients; see Theorem \ref{thm:main_hon-bis} and Corollary \ref{thm:main_hon-bis-corollary}. More in detail, for nonnegative $g$ we define
	\[
	\phi_g(\theta)=\left\{
	\begin{array}{cc}
	g(\theta), & \theta\in [0,2\pi], \\
	-g(-\theta), & \theta\in [-2\pi,0).  
	\end{array}
	\right.\,
	\]
	Our main result is that $\{Y_nT_n[f]\}_n$ is distributed as $\phi_g$ in the sense of eigenvalues with $g(\theta)=|f(\theta)|$. The secondary result resumed in Theorem \ref{thm:main_hon-bis-corollary2} is that the preconditioned matrix sequence $\{|C_n|^{-1}Y_nT_n[f]\}_n$ introduced in \cite{doi:10.1137/140974213} shows the spectral distribution $\phi_1$ independent of $f$ and the latter is equivalent to the second part of Theorem $4.1$ in \cite{HMS}.
	
	The spectral analysis of $\{Y_nT_n[f]\}_n$ is performed by using a general result of $2$-by-$2$  block matrix sequences, whose generality goes beyond the specific case under consideration. The other ingredient of our analysis is the notion of approximation class sequences introduced in the theory of GLT sequences (see the original definition in \cite{CAPIZZANO2001121} and several applications  in \cite{MR3674485}).
	
	Numerical experiments concerning different $T_n[f]$ and the corresponding circulant preconditioners are provided and critically discussed at the end of the paper.

\section{Preliminaries on Toeplitz matrices}\label{section:preliminaries}

As indicated in the introduction, we assume that the considered Toeplitz matrix $T_{n}[f] \in \mathbb{C}^{n \times n}$ is associated with a Lebesgue integrable function $f$ via its Fourier series
	\[
	f(\theta)\sim \sum_{k=-\infty}^{\infty}a_{k}e^{\mathbf{i}k\theta}
	\]
	defined on $[-\pi,\pi]$ and periodically extended on the whole real line. Thus, we have
	\[
	T_{n}[f]=\begin{bmatrix}{}
	a_0 & a_{-1} & \cdots & a_{-n+2} & a_{-n+1} \\
	a_1 & a_0 & a_{-1}   &  & a_{-n+2} \\
	\vdots & a_1 & a_0 & \ddots & \vdots \\
	a_{n-2} &  & \ddots & \ddots & a_{-1} \\
	a_{n-1} & a_{n-2} &\cdots & a_1 & a_0
	\end{bmatrix}, 
	\] where
	\[
	a_{k}=\frac{1}{2\pi} \int_{-\pi} ^{\pi}f(\theta) e^{-\mathbf{i} k \theta } \,d\theta,\quad k=0,\pm1,\pm2,\dots 
	\] 
	are the Fourier coefficients of $f$. The function $f$ is called the \emph{generating function} of $T_n[f]$. If $f$ is complex-valued, then $T_n[f]$ is non-Hermitian for all sufficiently large $n$. Conversely, if $f$ is real-valued, then $T_n[f]$ is Hermitian for all $n$. If $f$ is real-valued and nonnegative, but not identically zero almost everywhere, then $T_n[f]$ is Hermitian positive definite for all $n$. If $f$ is real-valued and even, $T_n[f]$ is (real) symmetric for all $n$ \cite{MR2108963,MR2376196}.

	The singular value and spectral distribution of Toeplitz matrix sequences has been well studied in the past few decades. Ever since Szeg{\H{o}} in \cite{MR890515} showed that the eigenvalues of the Toeplitz matrix $T_n[f]$ generated by a real-valued $f\in L^{\infty}([-\pi,\pi])$ are asymptotically distributed as $f$, such result has been undergone many generalizations and extensions. Under the same assumption on $f$, Avram and Parter \cite{MR952991,MR851935} proved that the singular values of $T_n[f]$ are distributed as $|f|$. Tyrtyshnikov \cite{Tyrtyshnikov19961,MR1258226,MR1481397} and Tilli \cite{MR1671591} later furthered the result for $T_n[f]$ generated by $f\in L^1([-\pi,\pi])$. Recently, Garoni, Serra-Capizzano, and Vassalos \cite{MR3399336} provided the same theorem based on the theory of Generalized Locally Toeplitz (GLT) sequences \cite{MR3674485}. The changes in the singular value and spectral distribution of Toeplitz matrix sequences after certain matrix operations were studied by Tyrtyshnikov and Serra-Capizzano respectively in \cite{TYRTYSHNIKOV1994225,CAPIZZANO2001121,glt-1,glt-2}.
	
		%The generalized Szeg{\H{o}} theorem that describes the singular value and spectral distribution of Toeplitz sequences is given as follows:
	
	\begin{theorem}\cite[Theorem 6.5]{MR3674485}\label{lem:main}
		Suppose $f \in L^{1}([-\pi,\pi])$. Let $T_n[f]$ be the Toeplitz matrix generated by $f$. Then 
		\[
		\{T_n[f]\}_n \sim_{\sigma} f.
		\]
		If moreover $f$ is real-valued, then
		\[
		\{T_n[f]\}_n \sim_{\lambda} f.
		\]
	\end{theorem}

	In the following, we always assume that $f\in L^1([-\pi, \pi])$ and is periodically extended to the real line. Furthermore, we follow all standard notation and terminology introduced in \cite{MR3674485}: let $C_c(\mathbb{C})$ (or $C_c(\mathbb{R})$) be the space of complex-valued continuous functions defined on $\mathbb{C}$ (or $\mathbb{R}$) with bounded support and let $\phi$ be a functional, i.e. any function defined on some vector space which takes values in $\mathbb{C}$. Also, if $g:D\subset \mathbb{R}^k \to \mathbb{K}$ ($\mathbb{R}$ or $\mathbb{C}$) is a measurable function defined on a set $D$ with $0<\mu_k(D)<\infty$, the functional $\phi_g$ is denoted such that
	\[
	\phi_g:C_c(\mathbb{K})\to \mathbb{C} \quad \text{and} \quad \phi_g(F)=\frac{1}{\mu_k(D)}\int_D F(g(\mathbf{x}))\,d\mathbf{x}.
	\]

\begin{definition}\cite[Definition 3.1]{MR3674485}(Singular value and eigenvalue distribution of a matrix sequence)\label{def:spectral_distribution}
Let $\{A_n\}_n$ be a matrix sequence.
		
\begin{enumerate}

\item We say that $\{A_n\}_n$ has an asymptotic singular value distribution described by a functional $\phi:C_c(\mathbb{R})\to \mathbb{C},$ and we write $\{A_n\}_n \sim_{\sigma}\phi,$ if
\[
\lim_{n\to \infty} \frac{1}{n}\sum_{j=1}^{n}F(\sigma_j(A_n))=\phi(F),\quad \forall F \in C_c(\mathbb{R}).
\] If $\phi=\phi_{|f|}$ for some measurable $f:D \subset \mathbb{R}^k \to \mathbb{C}$ defined on a set $D$ with $0<\mu_k(D)<\infty,$ we say that $\{A_n\}_n$ has an asymptotic singular value distribution described by $f$ and we write $\{A_n\}_n \sim_{\sigma} f.$% In this case, the function $f$ is referred to as the singular value symbol of the matrix-sequence $\{A_n\}_n$.

\item We say that $\{A_n\}_n$ has an asymptotic eigenvalue (or spectral) distribution described by a function $\phi:C_c(\mathbb{R})\to \mathbb{C},$ and we write $\{A_n\}_n \sim_{\lambda}\phi,$ if
\[
\lim_{n\to \infty} \frac{1}{n}\sum_{j=1}^{n}F(\lambda_j(A_n))=\phi(F),\quad \forall F \in C_c(\mathbb{C}).
\] If $\phi=\phi_{f}$ for some measurable $f:D \subset \mathbb{R}^k \to \mathbb{C}$ defined on a set $D$ with $0<\mu_k(D)<\infty,$ we say that $\{A_n\}_n$ has an asymptotic eigenvalue (or spectral) distribution described by  $f$ and we write $\{A_n\}_n \sim_{\lambda} f.$% In this case, the function $f$ is referred to as the eigenvalue (or spectral) symbol of the matrix-sequence $\{A_n\}_n$.
			\item Let $\{A_{n}\}_{n}$ be a matrix-sequence. We say that $\{A_{n}\}_{n}$ is
			\textit{sparsely vanishing (s.v.)} if for every $M>0$ there exists $n_{M}$
			such that, for $n\geq n_{M}$,
			\[
			\frac{\#\left\{ i\in \{1,...,n\}:\sigma _{i}(A_{n})<1/M\right\} }{n}
			\leq r(M)
			\]
			where\ $\lim_{M\rightarrow \infty }r(M)=0.$
			
			Note that $\{A_{n}\}_{n}$ is \textit{sparsely vanishing if and only} if
			\[
			\lim_{M\rightarrow \infty }\lim \sup_{n\rightarrow \infty }\frac{\#\left\{
				i\in \{1,...,n\}:\sigma _{i}(A_{n})<1/M\right\} }{n}=0,
			\]
			i.e.
			\[
			\lim_{M\rightarrow \infty }\lim \sup_{n\rightarrow \infty }\frac{1}{n}
			\sum_{i=1}^{n}\chi _{\lbrack 0,1/M)}\left( \sigma _{i}(A_{n})\right) =0.
			\]
			Finally we say that $\{A_{n}\}_{n}$ is \textit{sparsely vanishing (s.v.) in the sense of the eigenvalues} if in the previous two displayed equations the quantity $\sigma _{i}(A_{n})$ is replaced by $|\lambda _{i}(A_{n})|$ for $i=1,\ldots,n$.
		\end{enumerate}
	\end{definition}
	
	The following result holds (see a whole discussion on these issues in \cite[Chapter 9, pp. 165--166]{MR3674485}).
	
	\begin{theorem}\label{lem:sparsely vanishing}
		The followings are true.
		\begin{enumerate}
			\item
			Assume $\{A_n\}_n \sim_{\sigma} f.$  Then $\{A_n\}_n$ is sparsely vanishing if and only if $f$ is sparsely vanishing.
			\item
			Assume $\{A_n\}_n \sim_{\lambda} f.$  Then $\{A_n\}_n$ is sparsely vanishing in the eigenvalues sense if and only if $f$ is sparsely vanishing.
			\item Assume $\{A_n\}_n$ is given and assume that every matrix $A_n$ is normal. Then $\{A_n\}_n$ is sparsely vanishing if and only if  $\{A_n\}_n$ is sparsely vanishing in the eigenvalues sense.
		\end{enumerate}
	\end{theorem}

	Moreover, we introduce the following definitions and a key lemma in order to prove our main distribution results in the next chapter.
	
	\begin{definition}\cite[Definition 5.1]{MR3674485}(Approximating class of sequences)\label{def:ACS}
		Let $\{A_n\}_n$ be a matrix sequence and let $\{\{B_{n,m}\}_n\}_m$ be a sequence of matrix sequences. We say that $\{\{B_{n,m}\}_n\}_m$ is an \textit{approximating class of sequences (a.c.s)} for $\{A_n\}_n$ if the following condition is met: for every $m$ there exists $n_m$ such that, for $n \geq n_m$,
		\[
		A_n=B_{n,m}+R_{n,m}+N_{n,m},
		\]
		\[
		\text{rank}~R_{n,m}\leq c(m)n \quad \text{and} \quad \|N_{n,m}\|\leq\omega(m),
		\]
		where $n_m$, $c(m)$, and $\omega(m)$ depend only on $m$ and \[\lim_{m\to\infty}c(m)=\lim_{m\to\infty}\omega(m)=0.\]
	\end{definition}
	
	We use $\{B_{n,m}\}_n\xrightarrow{\text{a.c.s.\ wrt\ $m$}}\{A_n\}_n$ to denote that $\{\{B_{n,m}\}_n\}_m$ is an a.c.s for $\{A_n\}_n$.

	%Let $D \subset \mathbb{R}^k$ be a measurable set with  $0<\mu_k(D)<\infty$, and let
	%\[
	%\mathfrak{M}_D=\{f:D \to \mathbb{C}: f~\text{is measurable}\}.
	%\]
	
	\begin{definition}
		Let $f_m,f:D \subset \mathbb{R}^k \to \mathbb{C}$ be measurable functions. We say that $f_m \to f$ in measure if, for every $\epsilon > 0$,
		\[
		\lim_{m \to \infty} \mu_k\{|f_m-f|>\epsilon\}=0.
		\]
	\end{definition}

	\begin{lemma}\cite[Corollary 5.1]{MR3674485}\label{lem:Corollary5.1}
		Let $\{A_n\}_n, \{B_{n,m}\}_n$ be matrix sequences and let $f,f_m:D \subset \mathbb{R}^k \to \mathbb{C}$ be measurable functions defined on a set $D$ with $0<\mu_k(D)<\infty$. Suppose that
		
		\begin{enumerate}
			\item $\{B_{n,m}\}_n \sim_{\sigma}  f_m$ for every $m$,
			\item $\{B_{n,m}\}_n\xrightarrow{\text{a.c.s.\ wrt\ $m$}}\{A_n\}_n$,
			\item $f_m \to f$ in measure.
		\end{enumerate}
		
		Then 
		\[
		\{A_{n}\}_n \sim_{\sigma}  f.
		\]
		
		Moreover, if the first assumption is replaced by $\{B_{n,m}\}_n \sim_{\lambda}  f_m$ for every $m$, given that the other two assumptions are left unchanged, and all the involved matrices are Hermitian, then  $\{A_{n}\}_n \sim_{\lambda}  f$.
	\end{lemma}
	
	In the next theorem, the authors prove the asymptotic inertia of $Y_nT_n[f]$ that is an evaluation of the number of positive, negative, and zero eigenvalues.
	
	\begin{theorem}\cite[Theorem 4.1]{HMS}\label{thm:main_hon}
		Suppose $f \in L^1([-\pi,\pi])$ with real Fourier coefficients and $Y_n \in \mathbb{R}^{n \times n}$ is the anti-identity matrix. Let $T_n[f]\in \mathbb{R}^{n \times n}$ be the Toeplitz matrix generated by $f$. Then 
		\[
		\{Y_nT_n[f]\}_n \sim_{\sigma}  f.
		\] Moreover, $Y_nT_n[f]$ is (real) symmetric and if $f$ is sparsely vanishing then
		\[
		|n^{+}(Y_nT_n[f])-n^{-}({Y_nT_n[f]})|=o(n),
		\]
		with $n^{+}(\cdot)$ and $n^{-}(\cdot)$ denoting the number of positive and the negative eigenvalues of its argument, respectively. If in addition $f$ is a trigonometric polynomial and not identically zero, then
		\[
		|n^{+}(Y_nT_n[f])-n^{-}({Y_nT_n[f]})|=O(1),
		\]
		where the constant hidden in the big $O$ notation is two times the degree of the polynomial $f$.
	\end{theorem}
	
	To end this section, the following definition regarding circulant matrices is given which will be used in the proof of our results in the preconditioning setting.
	
	\begin{definition}\cite{doi:10.1137/140974213}\label{circ-modulus}
		For any circulant matrix $C_n \in \mathbb{C}^{n\times n}$, the \emph{absolute value circulant matrix} $|C_n|$ of $C_n$ is defined by
		\begin{eqnarray}\nonumber
		|C_n|&=& (C_n^* C_n)^{1/2}\\\nonumber
		&=&(C_n C_n^*)^{1/2}\\\nonumber
		&=&F_n|\Lambda_n|F_n^*,
		\end{eqnarray} where 
		$F_n = \left[\frac{\omega^{jk}}{\sqrt{n}} \right]_{j,k=0}^{n-1},~ \omega=e^{-\mathbf{i}{2\pi\over n}}$, and $|\Lambda_n|$ is the diagonal matrix in the eigendecomposition of $C_n$ with all entries replaced by their magnitude.
	\end{definition}
	
	\begin{remark}
		By definition, $|C_n|$ is Hermitian positive definite provided that $C_n$ is nonsingular.
	\end{remark}

	\section{Main results}\label{sec:main}
	
	We provide our main results on singular value and eigenvalue distribution in this section. 
	
	In  Theorem \ref{thm:main_hon-bis} and Corollary \ref{thm:main_hon-bis-corollary}, we furnish the eigenvalue distribution of $\{Y_nT_n[f]\}_n$, which can be used for deriving the second part of Theorem \ref{thm:main_hon} by using Cauchy interlacing arguments. Theorem \ref{thm:main_hon-bis}  is completely new and its derivation indicates a general argument whose importance goes far beyond the specific case. We give such a general result in Theorem \ref{thm:main_hon-gen}. 
	
	The section is concluded by Theorem \ref{thm:main_hon-bis-corollary2} on preconditioned matrix sequences and by a few comments and remarks on the impact of the results.

	\subsection{A general tool and the spectral results on $\{Y_nT_n[f]\}_n$}\label{ssec:main1}

	Given $D \subset \mathbb{R}^k$ with $0<\mu_k(D)<\infty$, we define $\tilde D$ as $D\bigcup D_p$, where $p\in  \mathbb{R}^k$ and $D_p=p+D$, with the constraint that $D$ and $D_p$ have non-intersecting interior part, that is $D^\circ \bigcap  D_p^\circ  =\emptyset$. In this way $\mu_k(\tilde D)=2\mu_k(D)$. Given any $g$ defined over $D$, we define $\psi_g$ over $\tilde D$ in the following manner
	\begin{equation}\label{def-psi}
	\psi_g(x)=\left\{
	\begin{array}{cc}
	g(x), & x\in D, \\
	-g(x-p), & x\in D_p, \ x \notin D.
	\end{array}
	\right.\,
	\end{equation}

	\begin{theorem}\label{thm:main_hon-gen}
		Suppose $k_n=o(n)$ with $k_n \in \mathbb{Z}$ and $A(n)\in \mathbb{C}^{(\lceil n/2\rceil+k_n) \times (\lfloor n/2\rfloor  -k_n)}$. Let $B_n, E_n \in \mathbb{C}^{n \times n}$ be Hermitian such that
		\[
		B_n=\left[\begin{array}{cc}
		O_{\lceil n/2\rceil+k_n}  & A(n) \\
		A(n)^* & O_{\lfloor n/2\rfloor -k_n} \end{array}\right]+E_n,
		\]
		with $O_{\lceil n/2\rceil+k_n}$ and $O_{\lfloor n/2\rfloor -k_n}$ being the square null matrices of size $\lceil n/2\rceil+k_n$ and $\lfloor n/2\rfloor -k_n$ respectively.
		If $\{A(n)\}_n \sim_{\sigma}  g$, where $g\ge 0$ is defined over $D$ with positive, finite Lebesgue measure, and $\{E_n\}_n \sim_{\sigma}  0$, then 
		\[
		\{B_n\}_n \sim_{\lambda}  \psi_g
		\] over the domain $\tilde D$, with $\psi_g$ as in (\ref{def-psi}).
	\end{theorem}
	
	\begin{proof}
		For the sake of notational simplicity we set $A=A(n)$ and we define the auxiliary matrix $G_n$ as follows
		\[
		G_n=\left[\begin{array}{cc}
		O_{\lceil n/2\rceil+k_n}  & A \\
		A^* & O_{\lfloor n/2\rfloor -k_n} \end{array}\right].
		\]
		
		Fixing $n$ and supposing $k_n \geq 0$, we define $m = \lfloor n/2\rfloor  -k_n$ and $M = \lceil n/2\rceil+k_n$. Then, we consider the (full) singular value decomposition of $A=U_M \Sigma V_m^*$, where $U_M,V_m$ are unitary matrices of size $M$ and $m$, respectively, and $\Sigma$ is the rectangular diagonal matrix containing the singular values $\sigma_1,\dots,\sigma_m$. Denote by $O_{M,m}$ the rectangular null matrix of size $M \times m$. We have
		\begin{equation}\label{eqn:factorizationCn}
		G_n=\left[\begin{array}{cc}
		U_M   & O_{M,m} \\
		O_{m,M} & V_m \end{array}\right]
		\left[\begin{array}{cc}
		O_M   & \Sigma \\
		\Sigma^T & O_m \end{array}\right]
		\left[\begin{array}{cc}
		U_M^*   & O_{M,m} \\
		O_{m,M} & V_m^* \end{array}\right]
		\end{equation}
		which is similar to
		\[
		S_n = \left[\begin{array}{cc}
		O_M   & \Sigma \\
		\Sigma^T & O_m \end{array}\right].
		\]
		
		Notice that the matrix $\Sigma$ can be written as
		\begin{equation}\label{eqn:SigmaTilde}
			\Sigma=\left[\begin{array}{c}
			\tilde{\Sigma}_m \\
			O_{k,m}
			\end{array}\right],
			\qquad \tilde{\Sigma}_m=\left[\begin{array}{ccc}
			\sigma_1 & & \\
			& \ddots & \\
			& & \sigma_m
			\end{array}\right],
			\qquad k=M-m,
		\end{equation}
		where $\Sigma= \tilde{\Sigma}_m$ if $k = 0$.
		Under the hypothesis that $k_n \geq 0$, if the fixed $n$ is even, the index $k$ is equal to $2k_n$. Otherwise, it is equal to $2k_n+1$. 
		
		Using (\ref{eqn:SigmaTilde}), the matrix $S_n$ can be written as
		\begin{eqnarray}\nonumber
		S_n &=& \left[\begin{array}{cc}
		O_M   & \Sigma \\
		\Sigma^T & O_m \end{array}\right]= \left[\begin{array}{ccc}
		O_{m} & O_{m,k} & \tilde{\Sigma}_m  \\
		O_{k,m} & O_{k} & O_{k,m} \\
		\tilde{\Sigma}_m & O_{m,k} & O_{m}
		\end{array}\right],
%\\ \nonumber
%		&=& \left[\begin{array}{ccc}
%		O_{m} & O_{m,k} & I_m \\
%		O_{k,m} & O_{k} & O_{k,m} \\
%		I_m & O_{m,k} & O_{m}
%		\end{array}\right] 
%		\underbrace{\left[\begin{array}{ccc}
%		O_m & O_{m,k} & \tilde{\Sigma}_m \\
%		O_{k,m} &O_{k} & O_{k,m} \\
%		\tilde{\Sigma}_m & O_{m,k} & O_m
%		\end{array}\right]}_{S_n}
%		\left[\begin{array}{ccc}
%		O_{m} & O_{m,k} & I_m \\
%		O_{k,m} & O_{k} & O_{k,m} \\
%		I_m & O_{m,k} & O_{m}
%		\end{array}\right] 
%		,
		\end{eqnarray}
		where, if $k = 0$, the central row and column are not present and which, up to similarity by an obvious permutation, can be written as the direct sum of $O_{k}$ and
\[
\left[\begin{array}{cc}
		O_{m} &  \tilde{\Sigma}_m  \\
		\tilde{\Sigma}_m &  O_{m}
		\end{array}\right].
\] 
The latter matrix is a $2\times 2$ block circulant and hence it can be diagonalized by the $2\times 2$ block Fourier matrix so that
\[
\left[\begin{array}{cc}
		O_{m} &  \tilde{\Sigma}_m  \\
		\tilde{\Sigma}_m &  O_{m}
		\end{array}\right]
= \frac{\sqrt{2}}{2} 
		 \left[\begin{array}{cc}
		I_{m} & I_m \\
		I_m &  -I_{m}
		\end{array}\right]
 \left[\begin{array}{cc}
		\tilde{\Sigma}_m & O_m \\
		O_m & -\tilde{\Sigma}_m
		\end{array}\right] 	
 \frac{\sqrt{2}}{2} 
		 \left[\begin{array}{cc}
		I_{m} & I_m \\
		I_m &  -I_{m}
		\end{array}\right].
\]
Therefore, putting together the above information, we can  write the factorization
		\[
		S_n= \left[\begin{array}{ccc}
		O_{m} & O_{m,k} &\tilde{\Sigma}_m\\
		O_{k,m} & O_{k} & O_{k,m} \\
		\tilde{\Sigma}_m & O_{m,k} & O_{m}
		\end{array}\right] 
		= 
		Q_n
		 \left[\begin{array}{ccc}
		\tilde{\Sigma}_m & O_{m,k} & O_m \\
		O_{k,m} & O_{k} & O_{k,m} \\
		O_m & O_{m,k} & -\tilde{\Sigma}_m
		\end{array}\right] 		
		Q_n,
		\]
		where $Q_n$ is the orthogonal matrix
		\[
		Q_n = \frac{\sqrt{2}}{2} 
		 \left[\begin{array}{ccc}
		I_{m} & O_{m,k} & I_m \\
		O_{k,m} & \sqrt{2}I_{k} & O_{k,m} \\
		I_m & O_{m,k} & -I_{m}
		\end{array}\right]
		\]
given by the direct sum of the identity of size $k$ and of the previous $2\times 2$ block Fourier matrix. Thus, we know that $G_n$ is similar to the block diagonal matrix
\begin{equation}\label{eqn:factorizationSigma}
\left[
\begin{array}{ccc}
\tilde{\Sigma}_m & O_{m,k} & O_m \\
O_{k,m} & O_{k} & O_{k,m} \\
O_m & O_{m,k} & -\tilde{\Sigma}_m
\end{array}
\right].
\end{equation}
and hence (\ref{eqn:factorizationSigma}) implies that we can write the eigenvalues of the matrix $G_n$ for the case $k_n \geq 0$. 
A similar factorization can be obtained for $k_n <0$, by defining $m = \lceil n/2\rceil+k_n$ and $M = \lfloor n/2\rfloor  -k_n$.
		
		In particular, the eigenvalues of $G_n$ are given by the set of the singular values of $A_n$, by the set of the negation of the singular values of $A_n$ and, in addition to these, at most $k = o(n)$ zero eigenvalues. From the latter, it is transparent that
		\[
                     \{G_n\}_n \sim_{\lambda}  \psi_g.
                 \]
		Finally, since all the involved matrices are Hermitian and the perturbation matrix sequence is zero distributed, i.e, $\{E_n\}_n \sim_{\lambda,\sigma}  0$, the desired result follows directly from the second part of Lemma \ref{lem:Corollary5.1}, taking into account that $\{\{G_{n}\}_n\}_m$ is a constant class of sequences (that is not depending on the variable $m$) and it is nevertheless an a.c.s for $\{B_n\}_n$.
	\end{proof}
	 We now employ Theorem \ref{thm:main_hon-gen} in the specific setting of symmetrized Toeplitz sequences.
	
	\begin{theorem}\label{thm:main_hon-bis}
		Suppose $f \in L^1([-\pi,\pi])$ with real Fourier coefficients and $Y_n \in \mathbb{R}^{n \times n}$ is the anti-identity matrix. Let $T_n[f]\in \mathbb{R}^{n \times n}$ be the Toeplitz matrix generated by $f$. Then 
		\[
		\{Y_nT_n[f]\}_n \sim_{\lambda}  \psi_{|f|}
		\]
		over the domain  $\tilde D$ with $D=[0,2\pi]$ and $p=-2\pi$.
	\end{theorem}
	
	\begin{proof}
		We let $H_\nu[f,-]$ be the $\nu$-by-$\nu$ Hankel matrix generated by $f$ containing the Fourier coefficients from $a_{-1}$ in the position $(1,1)$ to $a_{-2\nu+1}$ in the position $(\nu,\nu)$. Analogously, we let $H_\nu[f,+]$ be the $\nu$-by-$\nu$ Hankel matrix generated by $f$ containing the Fourier coefficients from $a_{1}$ in the position $(1,1)$ to $a_{2\nu-1}$ in the position $(\nu,\nu)$.
		
		We start by considering the case of even $n$ and writing $Y_nT_n[f]$ as a $2$-by-$2$  block matrix of size $n=2\nu$, i.e.
		\[
		Y_nT_n[f] = \left[\begin{array}{cc}
		Y_\nu H_\nu[f,+] Y_\nu   & Y_\nu T_\nu[f] \\
		Y_\nu T_\nu[f]  & H_\nu[f,-] \end{array}\right].
		\]
		Note that for Lebesgue integrable $f$, $H_\nu[f,+]$ is exactly the Hankel matrix generated by $f$ according to the definition given in \cite{FasinoTilli}: in that paper it was proven that $\{H_\nu[f,+]\}_n\sim_{\sigma}  0$. Since in our setting  $H_\nu[f,+]$  is symmetric for every $\nu$, it follows that  $\{H_\nu[f,+]\}_n\sim_{\lambda}  0$. Hence, with $Y_\nu$ being both symmetric and orthogonal, we deduce that the matrix is symmetric with the same singular values as $H_\nu[f,+]$.
		Therefore
		\[
		\{Y_\nu H_\nu[f,+] Y_\nu\}_n\sim_{\lambda,\sigma}  0.
		\]
		Similarly, we have
		\[
		\{ H_\nu[f,-] \}_n\sim_{\lambda,\sigma}  0
		\]
		since $H_\nu[f,-]=H_\nu[\bar f,+]$ and $\bar f$ (being the conjugate of $f$) is Lebesgue integrable if and only if $f$ is Lebesgue integrable.
		
		Therefore, the matrix sequence $\{Y_nT_n[f]\}_n$ can be written as the sum of the matrix sequence whose eigenvalues are clustered at zero   
		\[
		\{  E_n\}_n= \left\{
		\left[\begin{array}{cc}
		Y_\nu H_\nu[f,+] Y_\nu   &  O\\
		O  & H_\nu[f,-] \end{array}\right]\right\}_n
		\]
		and the matrix sequence
		\[
		\left\{
		\left[\begin{array}{cc}
		O  &  Y_\nu T_\nu[f] \\
		Y_\nu T_\nu[f]   & O \end{array}\right]
		\right\}_n
		\]
		whose eigenvalues are $\pm \sigma_j(Y_\nu T_\nu[f])=\pm \sigma_j( T_\nu[f])$,
		$j=1,\ldots,\nu$. 
		
		Hence, the claimed thesis follows from the general Theorem \ref{thm:main_hon-gen} with $g=|f|$, $A=A^*=A^T=
		Y_\nu T_\nu[f]$, and $k_n=0$.
		
		In the case where $n$ is odd, the analysis is of the same type as before with a few slight technical changes.
		
		By setting $\nu=\lfloor n/2 \rfloor$ and $\mu= \lceil n/2 \rceil$, we have
		\[
		Y_nT_n[f] = \left[\begin{array}{ccc}
		Y_\nu H_\nu[f \cdot e^{-\mathbf{i}\theta},+] Y_\nu & v   & Y_\nu T_\nu[f] \\
		v^T & a_0 & w^T \\
		Y_\nu T_\nu[f]  & w & H_\nu[f \cdot e^{\mathbf{i}\theta},-] \end{array}\right],
		\]
		provided that $n\neq 1$.
		Therefore, the matrix sequence $\{Y_nT_n[f]\}_n$ can be written as the sum of the matrix sequence whose eigenvalues are clustered at zero, that is $\{ E_n \}_n$, where $E_n=E_n'+E_n''$ with
		\[
		E_n' =
		\left[\begin{array}{cc}
		Y_\mu H_\mu[f\cdot e^{\mathbf{i}  \theta },+] Y_\mu   &  O\\
		O  & H_\nu[f \cdot e^{\mathbf{i}\theta},-] \end{array}\right],
		\]
		\[
		Y_\mu H_\mu[f\cdot e^{\mathbf{i}  \theta },+] Y_\mu =
		\left[\begin{array}{cc}
		Y_\nu H_\nu[f \cdot e^{-\mathbf{i}\theta},+] Y_\nu & v   \\
		v^T & a_0
		\end{array}\right],
		\]
		\[
		E_n'' =
		\left[\begin{array}{ccc}
		O & {\bf 0}   & O \\
		{\bf 0}^T & 0 & w^T \\
		O  & w & O \end{array}\right],
		\]
		and the matrix sequence
		\[
		\left\{
		\left[\begin{array}{ccc}
		O  &  {\bf 0} & Y_\nu T_\nu[f] \\
		{\bf 0}^T & 0 &  {\bf 0}^T \\
		Y_\nu T_\nu[f] & {\bf 0}   & O \end{array}\right]
		\right\}_n
		\]
		whose eigenvalues are $0$ with multiplicity $1$ and $\pm \sigma_j(Y_\nu T_\nu[f])$,
		$j=1,\ldots,\nu$. Note that we have $\sigma_j(Y_\nu T_\nu[f])= \sigma_j( T_\nu[f])$,
		$j=1,\ldots,\nu$, again from the singular value decomposition of $Y_\nu T_\nu[f]$, as in the proof of Theorem \ref{thm:main_hon-gen} when dealing with the matrix $G_n$ (see (\ref{eqn:factorizationSigma})).
		
		Consequently, the claimed thesis follows from the general Theorem \ref{thm:main_hon-gen} with $g=|f|$, 
		\[
		A= A(n)= \left[\begin{array}{c}
		Y_\nu T_\nu[f] \\
		{\bf 0}^T  \end{array}\right],\quad \quad 
		A^*=A(n)^*=A(n)^T=\left[\begin{array}{cc}
		Y_\nu T_\nu[f] &  {\bf 0} 
		\end{array}\right], 
		\]
		and $k_n=0$. 
	\end{proof}
	
	\begin{corollary}\label{thm:main_hon-bis-corollary}
		Suppose $f \in L^1([-\pi,\pi])$ with real Fourier coefficients and $Y_n \in \mathbb{R}^{n \times n}$ is the anti-identity matrix. Let $T_n[f]\in \mathbb{R}^{n \times n}$ be the Toeplitz matrix generated by $f$. Then, 
		\[
		\{Y_nT_n[f]\}_n \sim_{\lambda}  \phi_{|f|}
		\]
		over the domain  $[-2\pi,2\pi]$ with $\phi_g$ defined in the following way
		\[
		\phi_g(\theta)=\left\{
		\begin{array}{cc}
		g(\theta), & \theta\in [0,2\pi], \\
		-g(-\theta), & \theta\in [-2\pi,0).
		\end{array}
		\right.\,
		\]
	\end{corollary}
	
	\begin{proof}
		We observe that for any $F$ continuous with bounded support
		\[
		\int_{-2\pi}^{2\pi} F(\phi_{|f|})=\int_{-2\pi}^{2\pi} F(\psi_{|f|}),
		\]
		i.e. $\phi_{|f|}$ is a rearrangement of $\psi_{|f|}$ (and vice versa) \cite[Section 3.2]{MR3674485}. Hence, by the very definition of distribution, we have
		$\{Y_nT_n[f]\}_n \sim_{\lambda}  \phi_{|f|}$ if and only if $\{Y_nT_n[f]\}_n \sim_{\lambda}  \psi_{|f|}$. Therefore, the desired result is an immediate consequence of Theorem \ref{thm:main_hon-bis}. 
	\end{proof}
	
	Considering real-valued $f$, we remark that the spectral distribution of $\{Y_nT_n[f]\}_n$ is in stark contrast to that of $\{T_n[f]\}_n$ provided in Theorem \ref{lem:main} (the generalized Szeg{\H{o}} theorem), even though their singular value distributions are equivalent. Finally, the techniques given in this section can be adapted verbatim to the case of Toeplitz structures generated by $s\times s$ matrix-valued functions.

\begin{theorem}\label{thm:main_hon-ter}
		Suppose that the function $f$ is defined on $[-\pi,\pi]$ and is $s\times s$ matrix-valued. Assume that $f$ is Lebesgue integrable, i.e. $f_{j,k}\in L^1([-\pi,\pi])$, $j,k=1,\ldots s$, such that each Fourier coefficient of $f$ is a $s\times s$ Hermitian matrix and take $Y_n \in \mathbb{R}^{n \times n}$ as the anti-identity matrix. Let $T_{n,s}[f]\in \mathbb{C}^{sn \times sn}$ be the block Toeplitz matrix generated by $f$. Then 
\[
\{(Y_n\otimes I_s)T_{n,s}[f]\}_n \sim_{\lambda}  \psi_{|f|}, \quad\quad |f|=(f f^*)^{1/2},
\]
over the domain  $\tilde D$ with $D=[0,2\pi]$ and $p=-2\pi$ that is
\[
\lim_{n\to \infty} \frac{1}{sn}\sum_{j=1}^{sn}F(\lambda_j((Y_n\otimes I_s)T_{n,s}[f]))=
 \frac{1}{2\pi}\int_{-\pi}^{\pi} \frac{1}{s}\sum_{j=1}^{s} F(\lambda_j(|f|(\theta))\,d\theta.
\]
which is the generalization of the eigenvalue distribution in Item 2. of Definition \ref{def:spectral_distribution} for matrix-valued symbols with
\[
 \phi_{|f|} (F)=\frac{1}{2\pi}\int_{-\pi}^{\pi} \frac{1}{s}\sum_{j=1}^{s} F(\lambda_j(|f|(\theta)))\,d\theta.
\]
\end{theorem}
	
\subsection{Spectral results on preconditioned matrix sequences}\label{ssec:main2}
	
	In this subsection, we use the results of the previous subsection in order to deal with the eigenvalue distribution of certain preconditioned matrix sequences.
	
	\begin{theorem}\label{thm:main_hon-bis-corollary2}
		Suppose $f \in L^1([-\pi,\pi])$ with real Fourier coefficients and $Y_n \in \mathbb{R}^{n \times n}$ is the anti-identity matrix. Let $T_n[f]\in \mathbb{R}^{n \times n}$ be the Toeplitz matrix generated by $f$. Then 
		\[
		\{|C_n|^{-1} Y_nT_n[f]\}_n \sim_{\lambda}  \psi_{1}=\phi_1
		\]
		over the domain  $\tilde D$ with $D=[0,2\pi]$ and $p=-2\pi$ under the assumption that $\{C_n\}_n$ is a circulant matrix sequence such that
		\[
		\{C_n^{-1} T_n[f]\}_n \sim_{\sigma}  1.
		\]
		\end{theorem}
	\begin{proof}
		Because $|C_n|$ is positive definite as observed in the remark after Definition \ref{circ-modulus}, the matrices 
		\[
		|C_n|^{-1} Y_nT_n[f]
		\]
		and 
		\[
		|C_n|^{-1/2} Y_nT_n[f]|C_n|^{-1/2}
		\]
		are well defined and similar. They share the same eigenvalues clustered around $\{-1,1\}$ by \cite{doi:10.1137/140974213}, under the assumption that $\{C_n^{-1} T_n[f]\}_n$ is clustered around $1$ in the singular value sense. Also, by the Sylvester inertia law, the matrices 
		\[
		|C_n|^{-1/2} Y_nT_n[f]|C_n|^{-1/2} \quad \text{and} \quad Y_nT_n[f]
		\]
		have exactly the same inertia, namely the same number of positive, negative, and zero eigenvalues.
		However, by Theorem \ref{thm:main_hon}, we know that the matrix $Y_nT_n[f]$ has $n/2 +o(n)$ positive eigenvalues, $n/2 +o(n)$ negative eigenvalues, and $o(n)$ zero eigenvalues for large enough $n$.
		Therefore, by combining the above statements, we deduce that the matrix $|C_n|^{-1} Y_nT_n[f]$ possesses $n/2 +o(n)$  eigenvalues clustered around $1$ and $n/2 +o(n)$ eigenvalues clustered around $-1$.
		
		A simple check shows that the latter statement is equivalent to writing 
		\[
		\{|C_n|^{-1} Y_nT_n[f]\}_n \sim_{\lambda}  \psi_{1}=\phi_1
		\]
		over the domain  $\tilde D$ with $D=[0,2\pi]$ and $p=-2\pi$.
	\end{proof}
	
	We now complement the previous theorem with a short discussion regarding the hypothesis $\{C_n^{-1} T_n[f]\}_n \sim_{\sigma}  1$. Going back to the analysis in \cite{EstaSerra,koro-general}, we have the following picture:
	\begin{description}
		\item[A)] when $C_n$ is the Strang preconditioner for $T_n[f]$ and $X^+$ denotes the pseudo-inverse of $X$., the key assumption $\{C_n^{+} T_n[f]\}_n \sim_{\sigma}  1$ holds if $f$ is sparsely vanishing and belongs to the Dini-Lipschitz class (see for example \cite[Proposition 2.1, item 2]{EstaSerra}) which is a proper subset of the continuous $2\pi$-periodic functions;
		\item[B)] when $C_n$ is the Frobenius optimal preconditioner for $T_n[f]$, the key assumption $\{C_n^{+} T_n[f]\}_n \sim_{\sigma}  1$ holds if $f$ is sparsely vanishing and simply Lebesgue integrable (such a general result was proven quite elegantly by combining the Korovkin theory \cite{koro-general} and the GLT analysis \cite{MR3674485});
\item[C)]  By combining item {\bf A)} and item  {\bf B)}, we can update Theorem \ref{thm:main_hon-bis-corollary2}, by including the case where $C_n$ is not necessarily invertible. It is enough to replace $C_n^{-1}$ by $C_n^{+}$, taking into account that the assumption of $f$ sparsely vanishing will imply the presence of at most $o(n)$ zero eigenvalues both in the matrix $C_n$ and in the preconditioned matrix $C_n^{+} T_n[f]$.
	\end{description}
	
	The above statements cover the range of applicability of the preconditioned MINRES technique described in \cite{doi:10.1137/140974213}. Regarding the analysis wherein, it is worth observing that the matrix $\tilde C_n$ in \cite[Equation (3.4), page 276]{doi:10.1137/140974213} is not involutory as claimed in the paper. In fact, it is simply unitary: indeed its eigenvalues have unit modulus but in general they are not real. Hence, it is orthogonal when $C_n$ is real.
	
	\section{Numerical experiments}
	
	This section is divided into two subsections. In Subsection~\ref{num1}, we numerically show that the statements of Theorem~\ref{thm:main_hon-bis} and Corollary~\ref{thm:main_hon-bis-corollary} are true in the cases of both trigonometric polynomials and more generic functions in $L^1([-\pi,\pi])$. In Subsection~\ref{num2}, we illustrate the predicted behaviour of the eigenvalues of the preconditioned matrix sequences of Theorem~\ref{thm:main_hon-bis-corollary2} for different choices of generating functions and circulant preconditioners.
	
	\subsection{Numerical experiments on the spectral distribution of $\{Y_nT_n[f]\}_n$}\label{num1}
	
	In order to numerically support Theorem \ref{thm:main_hon-bis}, we show that for large enough $n$ the eigenvalues of $Y_nT_n[f]$ are approximately equal to the samples of $\psi_{|f|}$ over a uniform grid in $[-2\pi,2\pi]$, with the possible exception of a small number of outliers. We also remark that the function  $\phi_{|f|}$ of Corollary \ref{thm:main_hon-bis-corollary} has the same property, due to the rearrangement reason.
	
	We highlight the fact that the matrix $Y_nT_n[f]$ is symmetric for any $n$, so the quantities $\lambda_j(Y_nT_n[f])$ are real for $j=1,\dots,n$. In particular, we order the eigenvalues of $Y_nT_n[f]$ according to the evaluation of $\psi_{|f|}$ (respectively $\phi_{|f|}$) on the following uniform grid in $[-2\pi,2\pi]$:
	\begin{equation}\label{eq:uniform_grid}
		\theta_{j,n} = -2\pi+j\frac{4\pi}{n}, \qquad j = 1, \dots, n.
	\end{equation}
	
	Thus, in our experiments, we first compute the quantities $\psi_{|f|}(\theta_{j,n})$ (respectively $\phi_{|f|}(\theta_{j,n})$) for a fixed $n$ and then compare them with the properly sorted eigenvalues $\lambda_j(Y_nT_n[f])$, $j=1,\dots,n$.
		
	In Example~\ref{ex_poly_2_6__2_6_300}, we give numerical evidence of the fact that $\lambda_j(Y_nT_n[f])$ and $\psi_{|f|}(\theta_{j,n})$ are approximately equal for a real-valued, even trigonometric polynomial. In Example~\ref{ex_poly_4_2_2_9__4_1_300}, considering a trigonometric polynomial, we compare the quantities $\lambda_j(Y_nT_n[f])$ with both $\psi_{|f|}(\theta_{j,n})$ and $\phi_{|f|}(\theta_{j,n})$, and observe that they are approximately equal with the exception of $3$ outliers. In Example~\ref{ex_theta2_200}, we give numerical evidence of Theorem~\ref{thm:main_hon-bis} for a continuous function in $L^1([-\pi,\pi])$ and in Example~\ref{ex_pwc_80} we do the same for a discontinuous piecewise constant function in $L^1([-\pi,\pi])$.
	
	\begin{example}\label{ex_poly_2_6__2_6_300}
		We consider the real-valued, even trigonometric polynomial $f:[-\pi,\pi]\mapsto \mathbb{R}$ defined by
		\[
		f(\theta) = 2-12\cos(\theta).
		\]
		
		The $n$-by-$n$ Toeplitz matrix generated by $f$ is
		\[ 
		T_n[f] = \left[
		\begin{array}{cccc}
		2&-6&&\\
		-6&\ddots&\ddots&\\
		&\ddots&\ddots&-6\\
		&&-6&2 
		\end{array} \right].
		\]
		Notice that $T_n[f]$ is banded and symmetric, as we can see from the preliminaries on Toeplitz matrices in Section~\ref{section:preliminaries}.
		
		The multiplication by $Y_n$ produces the following matrix:
		\[ 
		Y_nT_n[f] = \left[
		\begin{array}{cccc}
		&&-6&2\\
		&\iddots&\iddots&-6\\
		-6&\iddots&\iddots&\\
		2&-6&& 
		\end{array} \right].
		\]
		The plot in Figure \ref{fig:ex_poly_2_6__2_6_300} shows that the eigenvalues of $Y_nT_n[f]$, properly sorted, are approximately equal to the samples of $\psi_{|f|}$ over $\theta_{j,n}$ for all $j = 1,\dots,n$. The plot is made for $n = 300$. This result is expected from the statement of Theorem~\ref{thm:main_hon-bis}. In this case, there are no outliers.
		
		\begin{figure}
			\centering
			\includegraphics[width=0.90\textwidth]{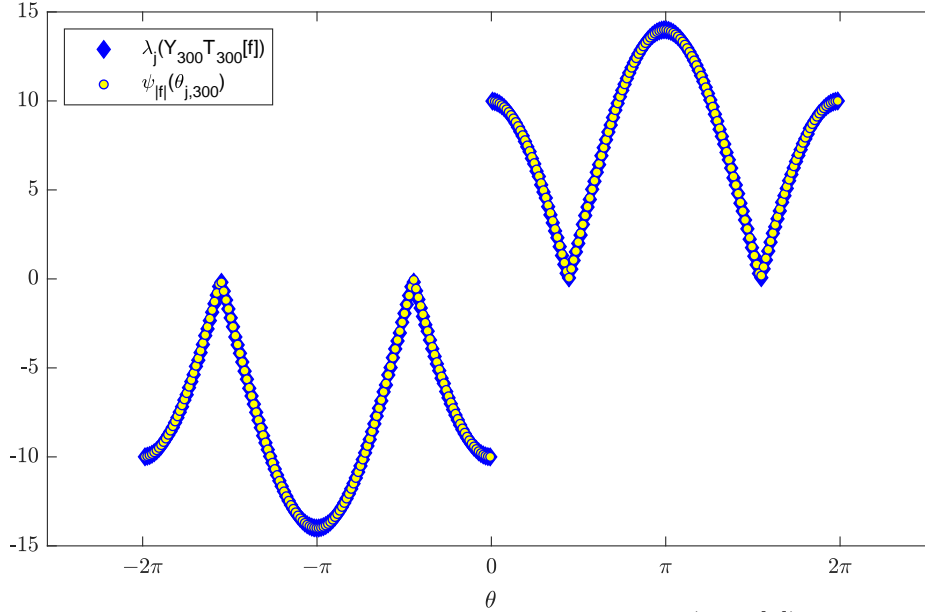}
			\vspace{-10pt}
			\caption{Example~\ref{ex_poly_2_6__2_6_300}, a comparison between the eigenvalues $\lambda_j(Y_nT_n[f])$ and the samples $\psi_{|f|}(\theta_{j,n})$, for $f(\theta) = 2-12\cos(\theta)$ and $n = 300$.}
			\label{fig:ex_poly_2_6__2_6_300}
		\end{figure}
	\end{example}
	
	\begin{example}\label{ex_poly_4_2_2_9__4_1_300}
		We consider the trigonometric polynomial $f:[-\pi,\pi]\mapsto \mathbb{C}$
		\[
		f(\theta) =4+2e^{-\mathbf{i}\theta}+2e^{-2\mathbf{i}\theta}+9e^{-3\mathbf{i}\theta}+e^{\mathbf{i}\theta}.
		\]
		
		The function $f$ generates a real, banded Toeplitz matrix $T_n[f]$. Differently from Example~\ref{ex_poly_2_6__2_6_300}, the matrix $T_n[f]$ in this case is not symmetric. However, the premultiplication by $Y_n$ produces the symmetric matrix $Y_nT_n[f]$ with real eigenvalues $\lambda_j(Y_nT_n[f])$.
		
		In this example, we compare the eigenvalues of $Y_nT_n[f]$ with the samples of $\psi_{|f|}$ (Figure \ref{fig:ex_poly_4_2_2_9__4_1_300}) and  $\phi_{|f|}$ (Figure \ref{fig:ex_poly_phi_4_2_2_9__4_1_300}) respectively. In both figures, we observe that the spectrum of $Y_nT_n[f]$ is well approximated by the evaluations of $\psi_{|f|}$ and  $\phi_{|f|}$ respectively, except for the presence of 3 outliers.
		
		The presence of such eigenvalues, which are not approximated by the sampling of $\psi_{|f|}$ and  $\phi_{|f|}$, is in line with the behaviour predicted by Theorem~\ref{thm:main_hon-bis} and Corollary~\ref{thm:main_hon-bis-corollary}. In fact, this agrees well with the concept of spectral distribution formalized in Definition~\ref{def:spectral_distribution}.
		\begin{figure}
			\centering
			\includegraphics[width=0.90\textwidth]{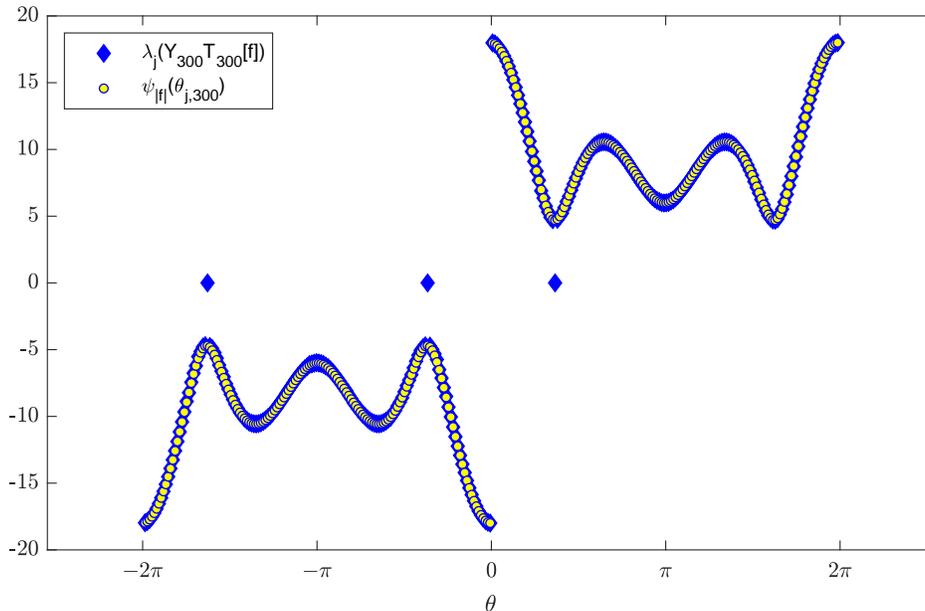}
			\vspace{-10pt}
			\caption{Example~\ref{ex_poly_4_2_2_9__4_1_300}, a comparison between the eigenvalues $\lambda_j(Y_nT_n[f])$ and the samples $\psi_{|f|}(\theta_{j,n})$, for $f(\theta) = 4+2e^{-\mathbf{i}\theta}+2e^{-2\mathbf{i}\theta}+9e^{-3\mathbf{i}\theta}+e^{\mathbf{i}\theta}$ and $n = 300$.}
			\label{fig:ex_poly_4_2_2_9__4_1_300}
		\end{figure}
		\begin{figure}
			\centering
			\includegraphics[width=0.90\textwidth]{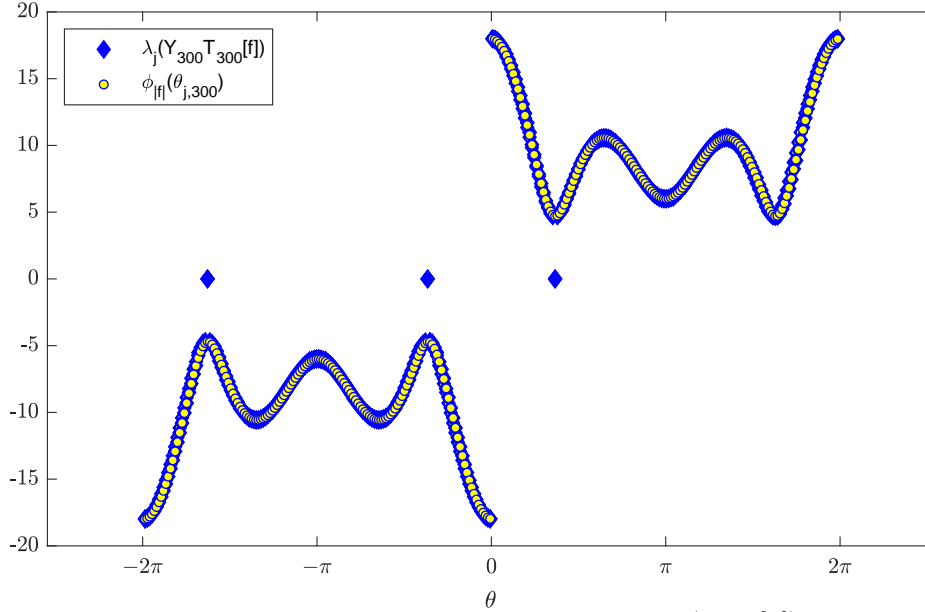}
			\vspace{-10pt}
			\caption{Example~\ref{ex_poly_4_2_2_9__4_1_300}, a comparison between the eigenvalues $\lambda_j(Y_nT_n[f])$ and the samples $\phi_{|f|}(\theta_{j,n})$, for $f(\theta) = 4+2e^{-\mathbf{i}\theta}+2e^{-2\mathbf{i}\theta}+9e^{-3\mathbf{i}\theta}+e^{\mathbf{i}\theta}$, and $n = 300$.}
			\label{fig:ex_poly_phi_4_2_2_9__4_1_300}
		\end{figure}
	\end{example}
	
	\begin{example}\label{ex_theta2_200}
		Let us define the function $f:[-\pi,\pi] \rightarrow \mathbb{R}$ by
		\[
		f(\theta) = \theta^2,
		\]
		periodically extended to the real line.
		
		The function $f$ is not a trigonometric polynomial, and consequently the matrices $T_n[f]$ are dense. In fact, the Fourier coefficients of $f$ are given by the formula
		\[
		\left\{
		\begin{array}{ll}
		a_0 = \frac{\pi^2}{3}, &  \\
		a_{k}=(-1)^k\frac{2}{k^2}, & k=\pm1,\pm2,\dots.
		\end{array}
		\right. 
		\]
		This expression can be derived by a direct computation of the quantities
		\[
		a_{k} = \frac{1}{\pi} \int_{0} ^{\pi} \theta^2 \cos(-\mathbf{i} k \theta) \,d\theta.
		\]
		
		In this example, we set $n$ equal to 200. We want to evaluate $\psi_{|f|}$ on the points of the grid $\theta_{j,n}$. Recalling that $f$ is defined on $[-\pi,\pi]$ and periodically extended to the real line, we can write an explicit formula for $f$ in $[0,2\pi]$:
		\[
		\left\{
		\begin{array}{ll}
		\theta^2, & \theta\in [0,\pi], \\
		(\theta-2\pi)^2, & \theta\in (\pi,2\pi].
		\end{array}
		\right. 
		\]	
		
		As a consequence of the definition of $f$, we have that the associated function $\psi_{|f|}$ is piecewise defined in the following 4 subintervals
		\[
		\psi_{|f|}(\theta_{j,n}) = 
		\left\{
		\begin{array}{cl}
		-(\theta_{j,n}+2\pi)^2, & \forall j = 1,\dots,\frac{n}{4},  \\
		-(\theta_{j,n})^2, & \forall j = \frac{n}{4}+1,\dots,\frac{n}{2},  \\
		(\theta_{j,n})^2, & \forall j = \frac{n}{2}+1,\dots,\frac{3n}{4},  \\
		(\theta_{j,n}-2\pi)^2, & \forall j = \frac{3n}{4}+1,\dots,n.
		\end{array}
		\right. 
		\]	
		
		In Figure~\ref{fig:ex_theta2_200}, we numerically show that the quantities $\psi_{|f|}(\theta_{j,n})$ approximate the eigenvalues $\lambda_j(Y_nT_n[f])$ for all $j = 1,\dots,n$. This result is expected from Theorem \ref{thm:main_hon-bis}, which holds for generic functions in $L^1([-\pi,\pi])$ with real Fourier coefficients.
		\begin{figure}
			\centering
			\includegraphics[width=0.90\textwidth]{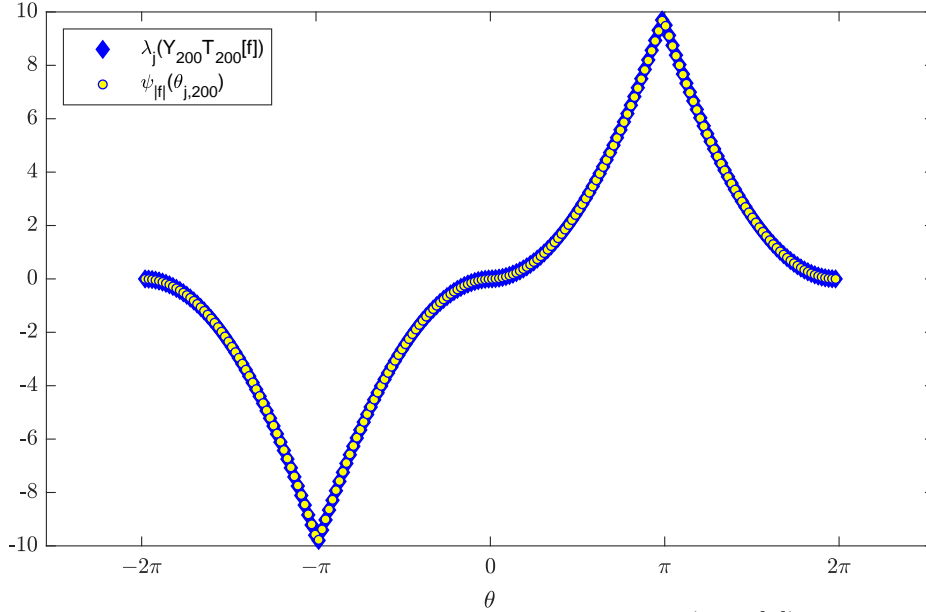}
			\vspace{-10pt}
			\caption{Example~\ref{ex_theta2_200}, a comparison between the eigenvalues $\lambda_j(Y_nT_n[f])$ and the samples $\psi_{|f|}(\theta_{j,n})$, for $f(\theta) = \theta^2$ and $n = 200$.}
			\label{fig:ex_theta2_200}
		\end{figure}
	\end{example}
	
	\begin{example}\label{ex_pwc_80}
		In the current example, we give numerical evidence of the distribution result of Theorem~\ref{thm:main_hon-bis} under the hypothesis that $f$ is a discontinuous function $f:[-\pi,\pi] \rightarrow \mathbb{R}$, piecewisely defined by the formula
		\[
		f(\theta)=\left\{
		\begin{array}{ll}
		5, & \theta\in [-\pi,-\pi/2), \\
		2, & \theta\in [-\pi/2,\pi/2), \\ 
		5, & \theta\in [\pi/2,\pi], \\
		\end{array}
		\right.
		\]
		and periodically extended to the real line. 
		
		We fix $n = 80$ and compute $\psi_{|f|}$ on the whole grid $\theta_{j,n}$ with a procedure similar to that in Example~\ref{ex_theta2_200}. In Figure~\ref{fig:ex_pwc_80} we show that the sampling $\psi_{|f|}(\theta_{j,n})$ is an approximation of the eigenvalues of the matrix $Y_nT_n[f]$ up to a constant number of outliers. 
		\begin{figure}
			\centering
			\includegraphics[width=0.90\textwidth]{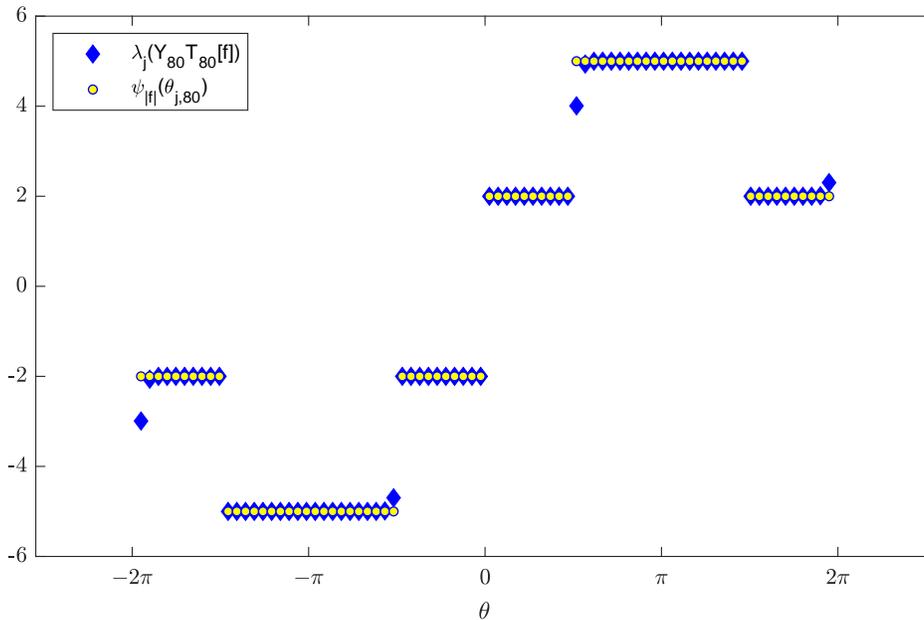}
			\vspace{-10pt}
			\caption{Example~\ref{ex_pwc_80}, comparison between the eigenvalues $\lambda_j(Y_nT_n[f])$ and the samples $\psi_{|f|}(\theta_{j,n})$, for the piecewise constant $f$ and $n = 80$.}
			\label{fig:ex_pwc_80}
		\end{figure}
	\end{example}

	\begin{example}\label{ex_block_Toeplitz}
	
	The last example of this subsection is the distribution result of the following matrix-valued function $f:[-\pi,\pi]\mapsto \mathbb{R}^{2 \times 2}$
	\[
	f(\theta)
	=
	\frac{1}{\sqrt{2}}
	\left[
	\begin{array}{cc}
	1   &  1\\
	1  & -1
	\end{array}
	\right]
	\left[
	\begin{array}{cc}
	10+2\cos{\theta} &  0\\
	0  & 2-\cos{\theta} 
	\end{array}
	\right]
	\frac{1}{\sqrt{2}}
	\left[
	\begin{array}{cc}
	1   &  1\\
	1  & -1
	\end{array}
	\right].
	\]
	
	Choosing $n = 200$, we compute $\psi_{|f|}$ on the uniform grid $\theta_{j,n}$ as before. Figure~\ref{fig:ex_block_Toeplitz} shows the sampling $\psi_{|f|}(\theta_{j,n})$ approximates the eigenvalues of the matrix $(Y_{n}\otimes I_s )T_{n,s}[f]$ well. We observe the four branches of eigenvalues $[-12,-8]\cup[-3,-1]\cup[1,3]\cup[8,12]$ as described by Theorem~\ref{thm:main_hon-ter}.
	
	\begin{figure}
	\centering
	\includegraphics[width=0.90\textwidth]{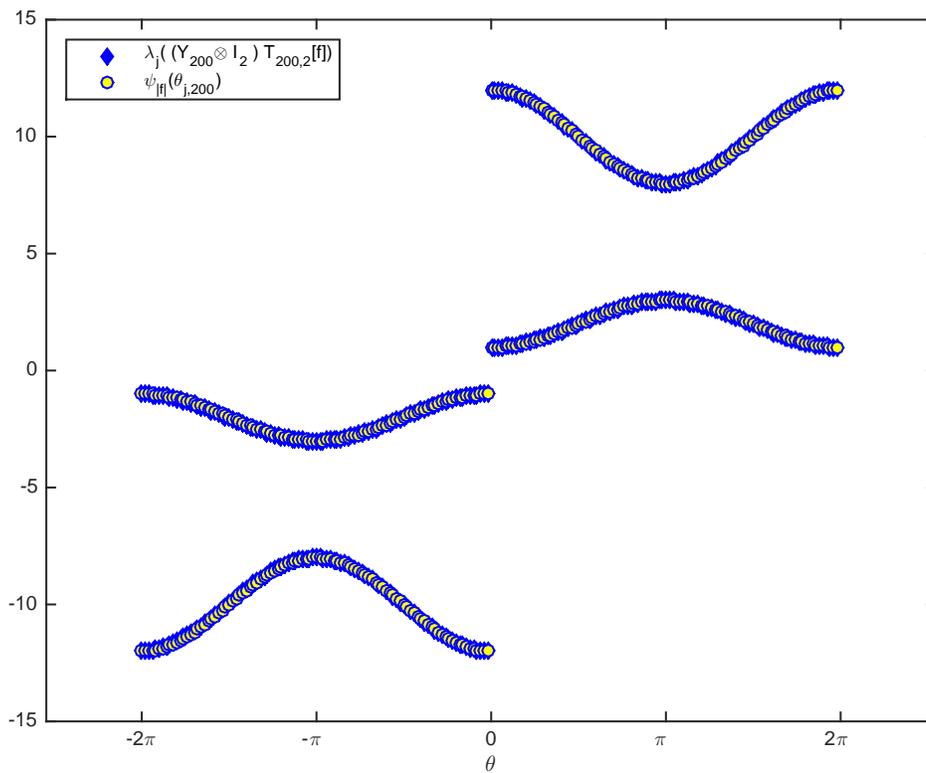}
	\vspace{-10pt}
	\caption{Example~\ref{ex_block_Toeplitz}, comparison between the eigenvalues $\lambda_j((Y_n \otimes I_s)T_{n,s}[f])$ and the samples $\psi_{|f|}(\theta_{j,n})$, for the matrix-valued $f$ and $(n,s) = (200,2)$.}
			\label{fig:ex_block_Toeplitz}
	\end{figure}
	
	\end{example}

	\subsection{Numerical experiments on preconditioned matrix sequences}\label{num2}
	
	This second subsection is dedicated to numerically illustrating the spectral behaviour of the preconditioned matrix sequence $\{|C_n|^{-1} Y_nT_n[f]\}_n$ as predicted in Theorem \ref{thm:main_hon-bis-corollary2}.
	
	Having proved that, under certain conditions, roughly half of the eigenvalues of $\{|C_n|^{-1} Y_nT_n[f]\}_n$ are clustered around $1$ and the other half around $-1$, we illustrate this spectral behaviour in several examples in the following.
	
	In particular, in Example \ref{ex_prec_poly} we focus on $f$ being a trigonometric polynomial. In Example~\ref{ex_prec_theta2}, we fix $f$ to be a quadratic function and in Example~\ref{ex_prec_pwc} we take $f$ as a discontinuous piecewise constant function.
	
	In the following examples, we first verify that the condition $\{C_n^{+} T_n[f]\}_n \sim_{\sigma}  1$ holds for each choice of generating function $f$ and the circulant preconditioner $C_n$. We prove this either using the discussion after Theorem \ref{thm:main_hon-bis-corollary2} (Examples~\ref{ex_prec_poly} and~\ref{ex_prec_theta2}) or numerically (Example~\ref{ex_prec_pwc}).
	
	Once that hypothesis is verified, we graphically show that the eigenvalues of $\{|C_n|^{-1} Y_nT_n[f]\}_n$ are distributed as the function $\psi_1$ over $[-2\pi,2\pi]$.
	
	In many cases, the greatest eigenvalue is an outlier and it becomes large very quickly. In order to make the figures more readable, we do not plot it.
	
	\begin{example}\label{ex_prec_poly}
		We consider the trigonometric polynomial
		\[
		f(\theta) =2-2e^{-\mathbf{i}\theta}-3e^{\mathbf{i}\theta}.
		\]
		
		Since $f$ is a nonzero polynomial, it is obviously sparsely vanishing and belongs to the Dini-Lipschitz class. Thus, we can use either the argument $\mathbf{A}$ or the argument $\mathbf{B}$ after Theorem \ref{thm:main_hon-bis-corollary2} to realize that $\{C_n^{+} T_n[f]\}_n \sim_{\sigma}  1$. We follow the argument $\mathbf{A}$ (the argument $\mathbf{B}$ is analogous), choosing as $C_n$ the Strang preconditioner for $T_n[f]$.
		
		 In Figure~\ref{fig:prec_poly}, we plot the eigenvalues of $|C_n|^{-1} Y_nT_n[f]$ for different values of $n$. For both $n = 500$ and $n = 1000$ we observe that the values $\lambda_j(|C_n|^{-1} Y_nT_n[f])$ are distributed as the function $\psi_{1}$, as predicted by Theorem \ref{thm:main_hon-bis-corollary2}. In fact, except for a constant number of outliers, half of the eigenvalues are equal to -1 and half of the eigenvalues are equal to 1.
		\begin{figure}
			\centering
			\includegraphics[width=0.45\textwidth]{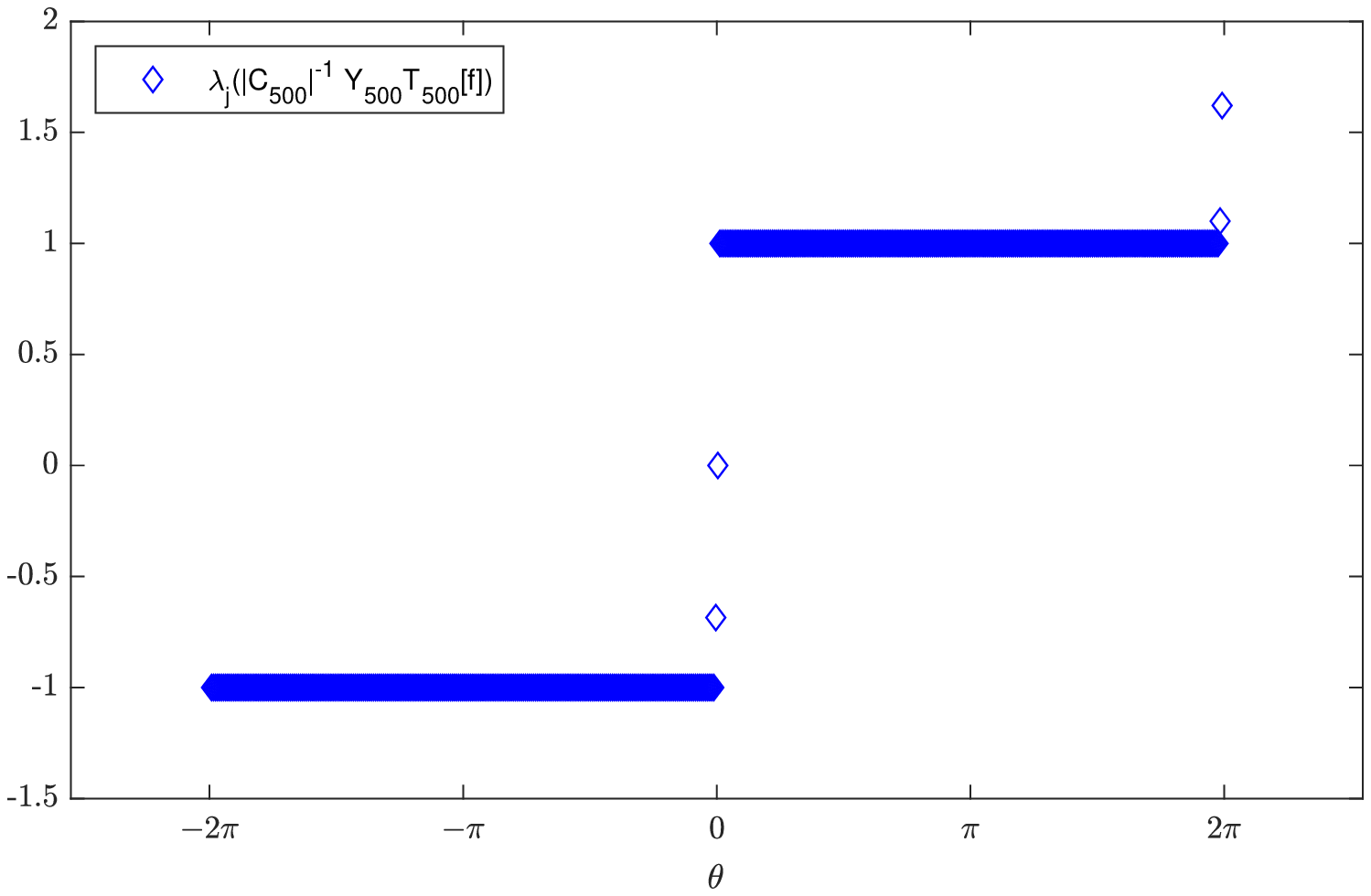}
			\includegraphics[width=0.45\textwidth]{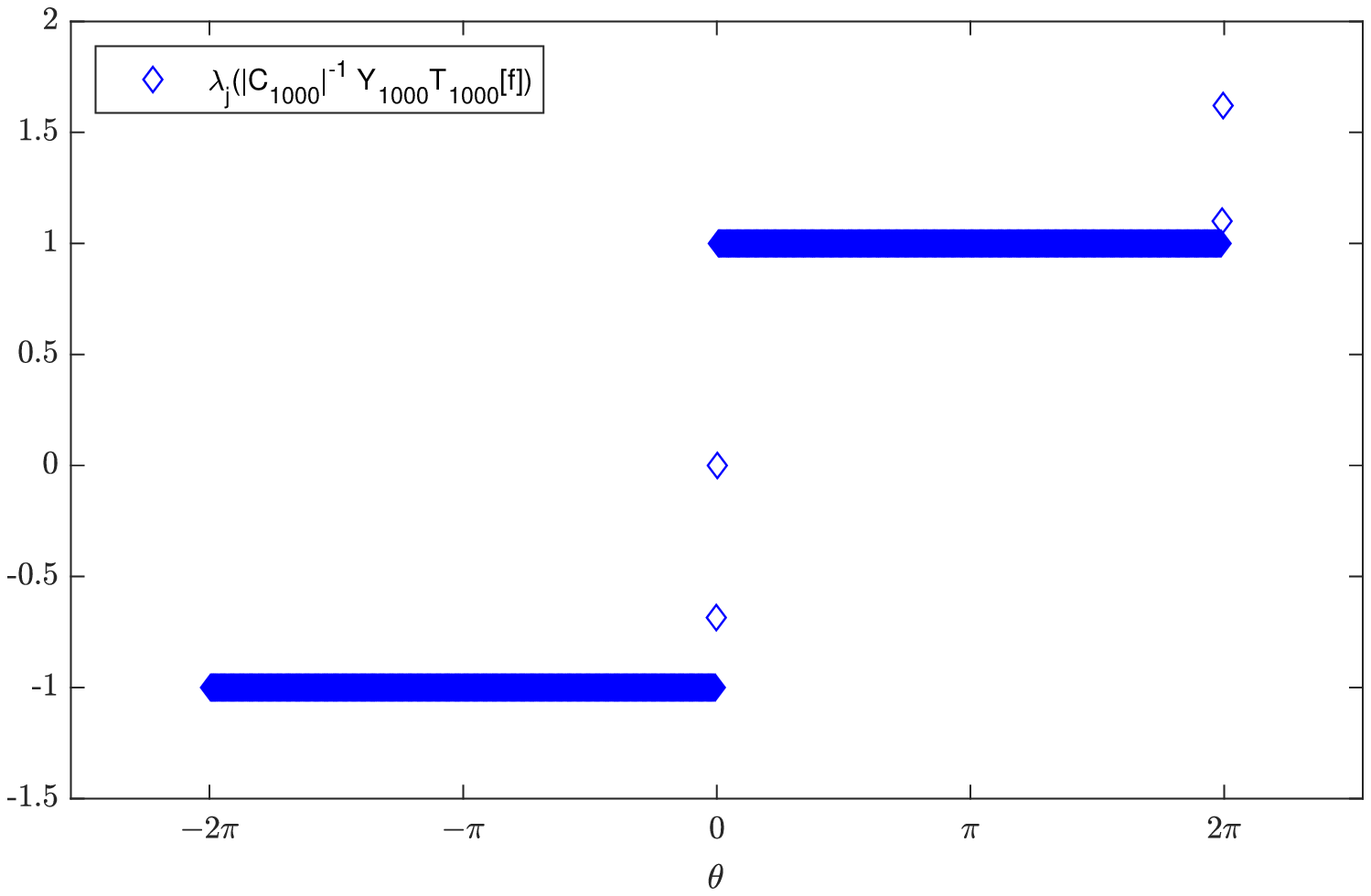}
			\caption{Example~\ref{ex_prec_poly}, the eigenvalues of $|C_n|^{-1} Y_nT_n[f]$, where $f(\theta) = 2-2e^{-\mathbf{i}\theta}-3e^{\mathbf{i}\theta}$, $C_n$ is the Strang preconditioner, and $n = 500, 1000$.}
			\label{fig:prec_poly}
		\end{figure}
	\end{example}
	
	\begin{example}\label{ex_prec_theta2}
		We consider the generating function
		\[
		f(\theta) = \theta^2.
		\]
		
		The remarks after Theorem \ref{thm:main_hon-bis-corollary2} assure us that, in this case, we can use both the Strang preconditioner and the Frobenius optimal preconditioner.
		
		For the current example, we show the results obtained from the two types of preconditioners, for different choices of $n$.
		
		In Figure~\ref{fig:prec_strang_theta2}, we plot the eigenvalues $\lambda_j(|C_n|^{-1} Y_nT_n[f])$, where $C_n$ is the Strang preconditioner for $n = 157, 200, 589, 1000$. For all tested $n$, the greatest eigenvalue $\lambda_n(|C_n|^{-1} Y_nT_n[f])$ is an outlier and becomes large quickly as $n$ increases. Consequently, this large outlier is not plotted for a better visualization of the values $\lambda_j(|C_n|^{-1} Y_nT_n[f])$ for $j = 1, \dots, n-1$.
		
		Notice that the spectrum of $|C_n|^{-1} Y_nT_n[f]$ is divided into two sets of almost the same cardinality: the first contains the eigenvalues equal to -1 and the second, instead, the eigenvalues equal to 1. Finally, the number of outliers that do not belong to the previous group is infinitesimal in the dimension $n$ of the matrix.
		
		In Figure~\ref{fig:prec_fro_theta2}, an analogous clustering of eigenvalues is shown using the Frobenius preconditioner for $n = 157, 200, 589, 1000$.  In this second experiment the Frobenius preconditioner gives us a worse result in terms of outliers. In fact, the number of outliers is significantly larger than that in the Strang preconditioner case. However, it is still infinitesimal with respect to $n$ as expected from the thesis of Theorem \ref{thm:main_hon-bis-corollary2}.
		\begin{figure}
			\centering
			\includegraphics[width=0.45\textwidth]{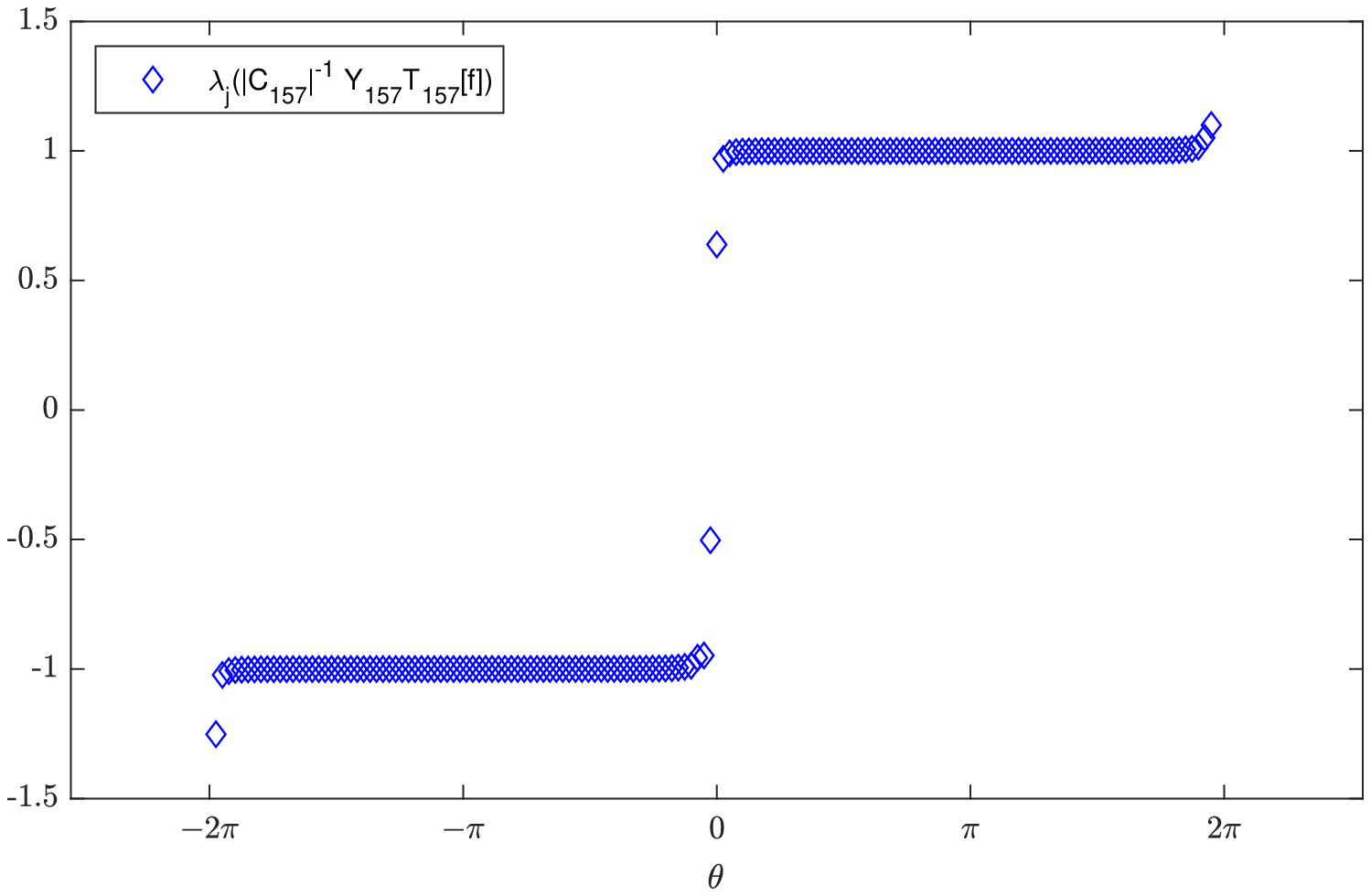}
			\includegraphics[width=0.45\textwidth]{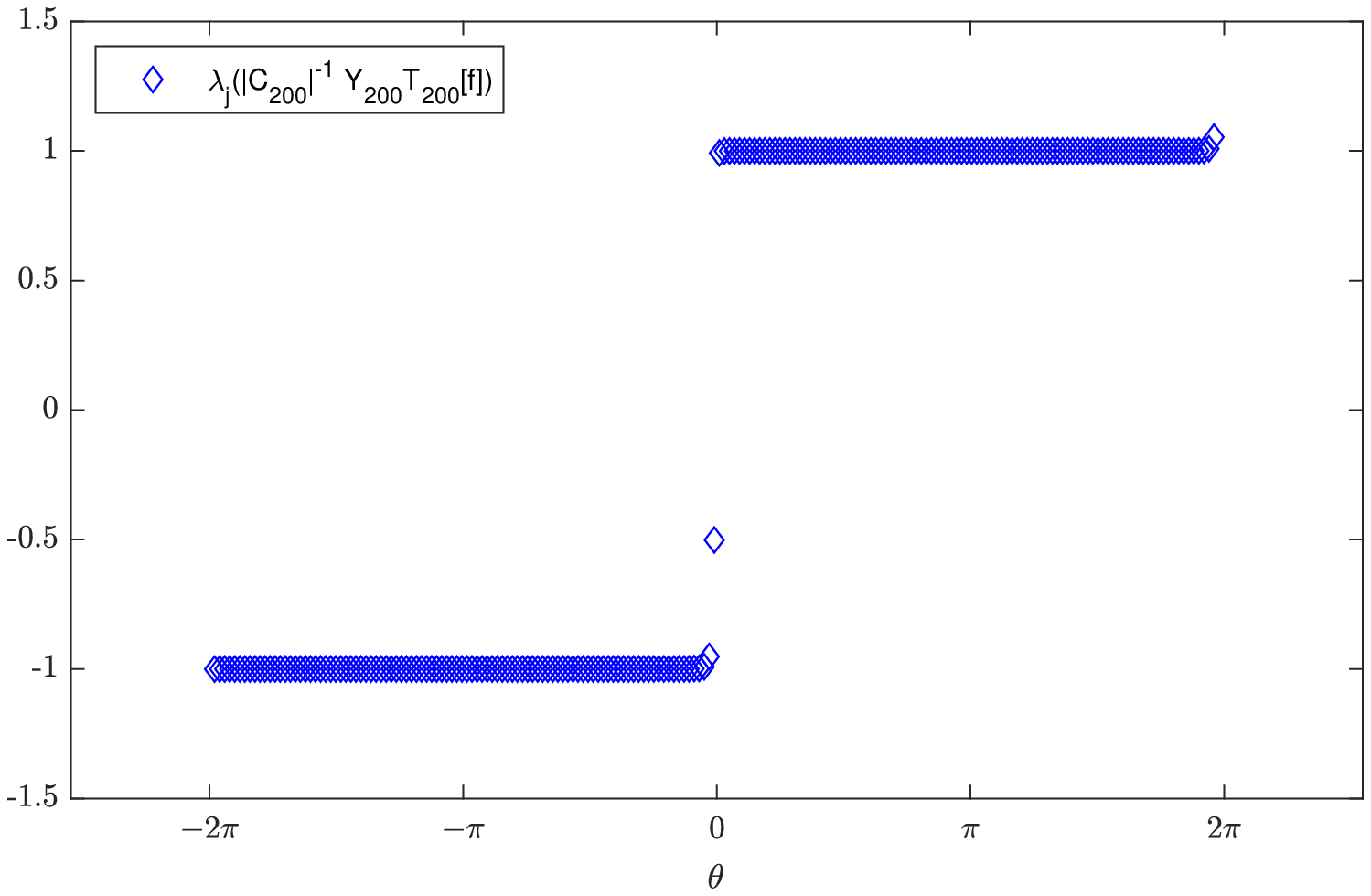}
			\includegraphics[width=0.45\textwidth]{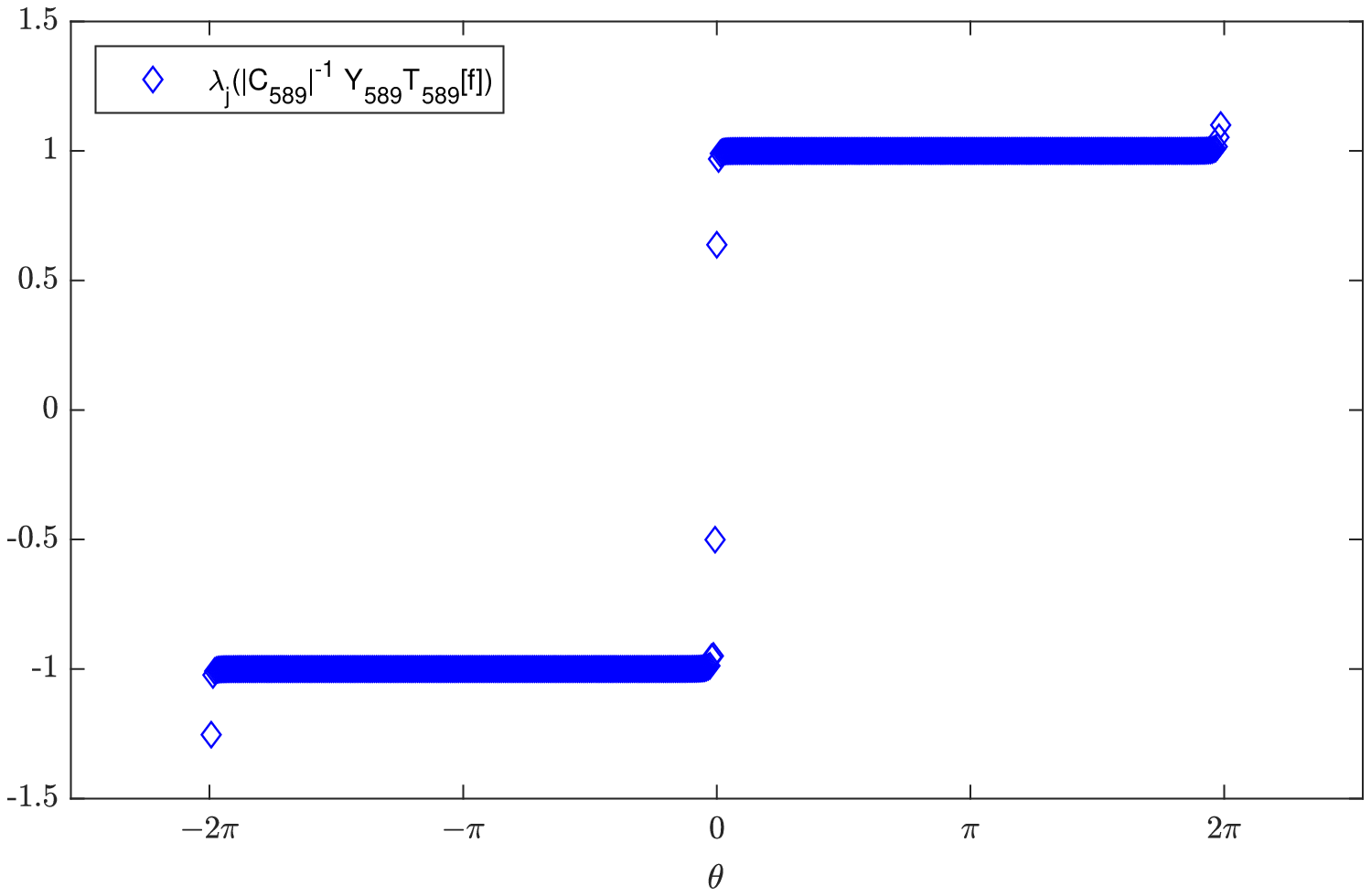}
			\includegraphics[width=0.45\textwidth]{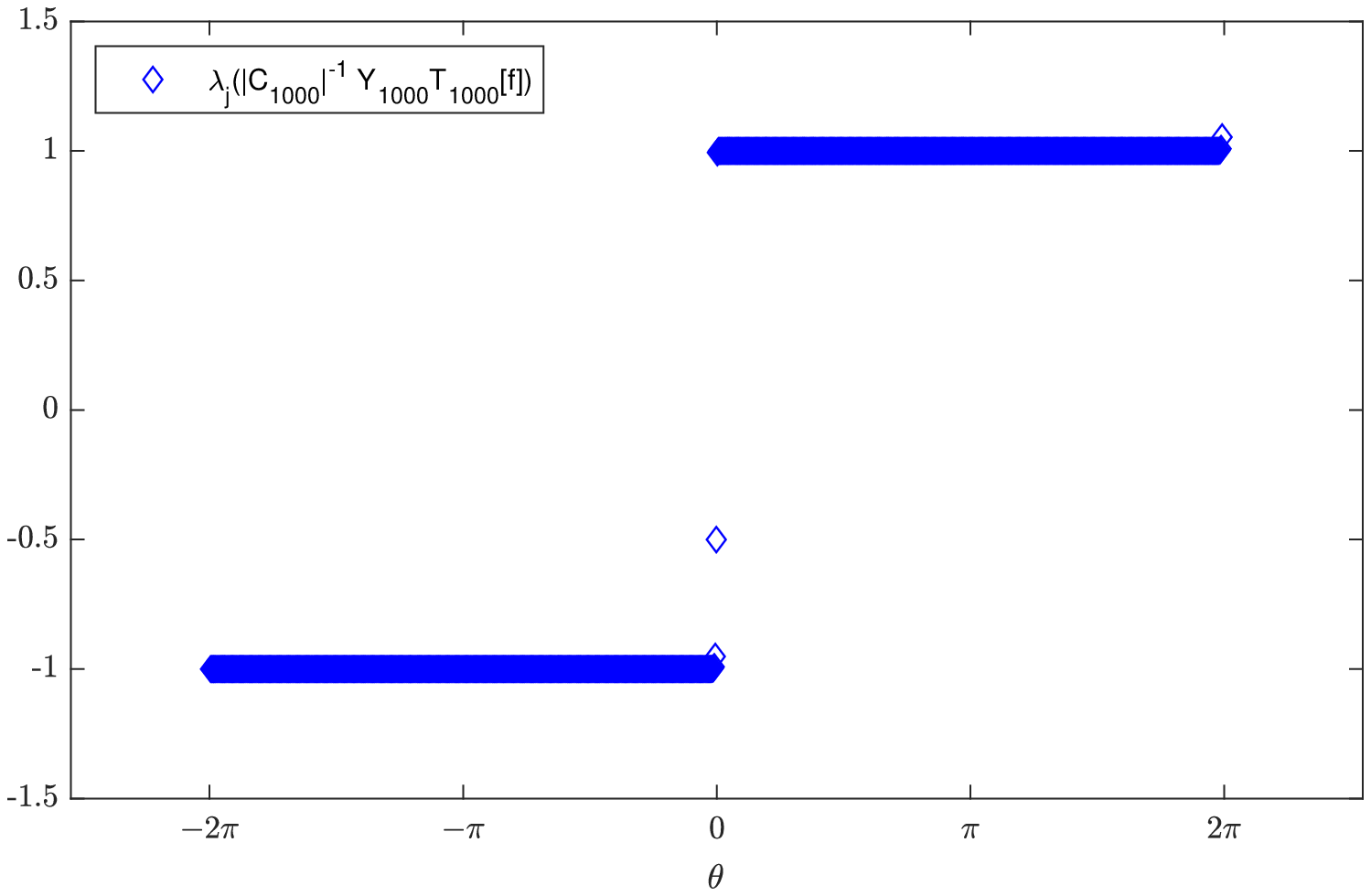}
			\caption{Example~\ref{ex_prec_theta2}, the eigenvalues of $|C_n|^{-1} Y_nT_n[f]$, where $f(\theta) = \theta^2$, $C_n$ is the Strang preconditioner, and $n = 157, 200, 589, 1000$. The greatest eigenvalue $\lambda_n(|C_n|^{-1} Y_nT_n[f])$ is not plotted.}
			\label{fig:prec_strang_theta2}
		\end{figure}
		\begin{figure}
			\centering
			\includegraphics[width=0.45\textwidth]{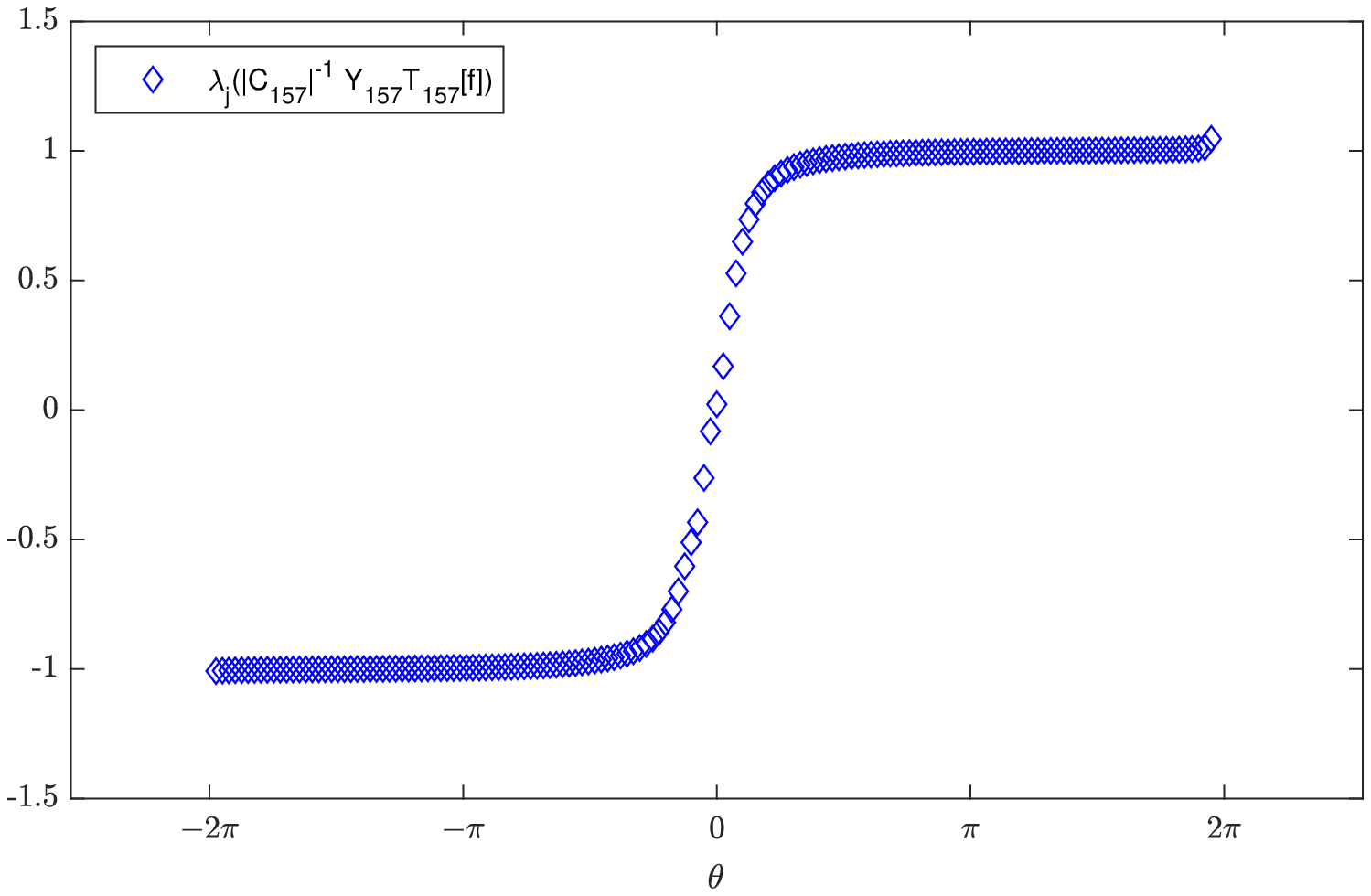}
			\includegraphics[width=0.45\textwidth]{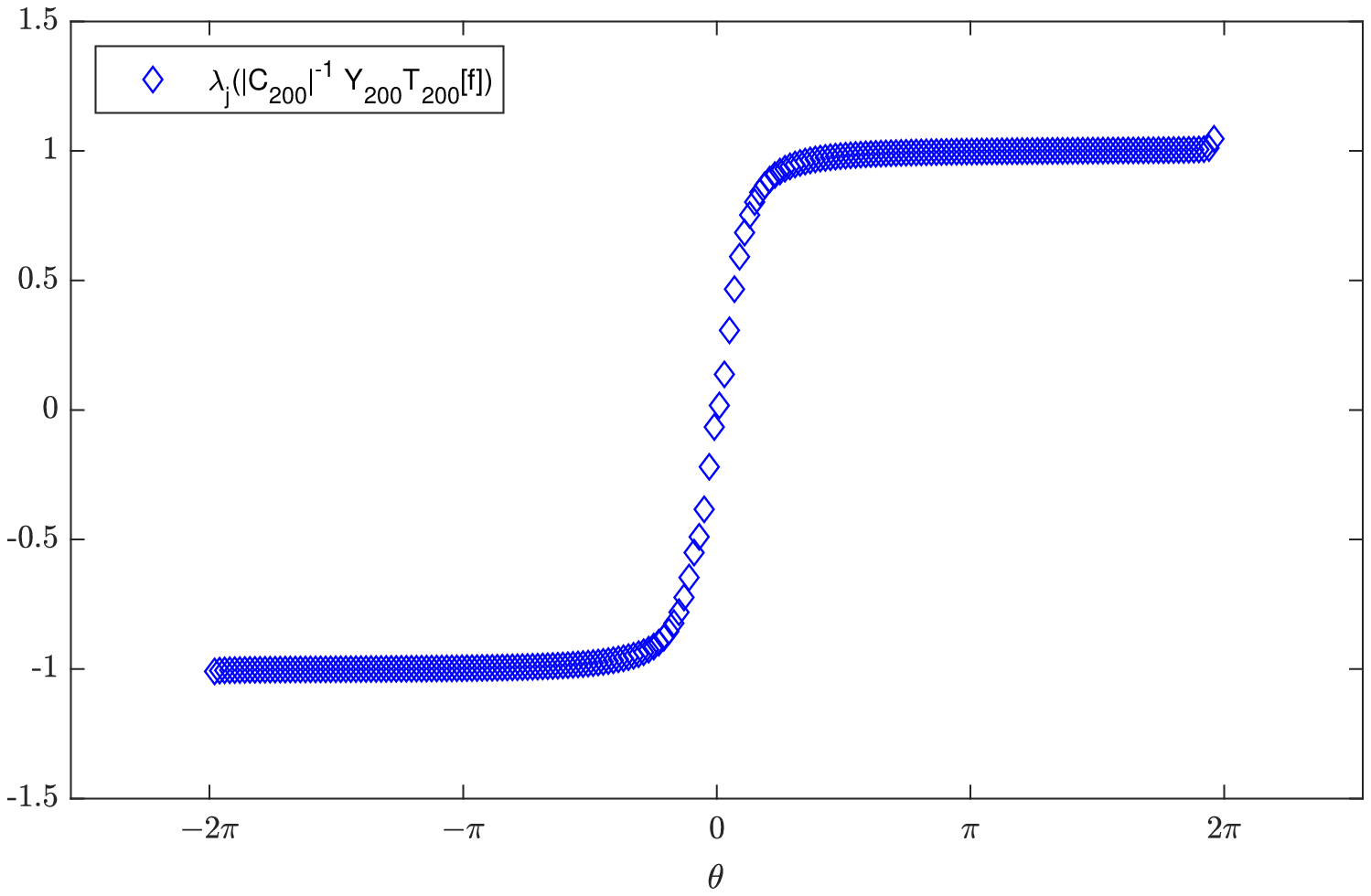}
			\includegraphics[width=0.45\textwidth]{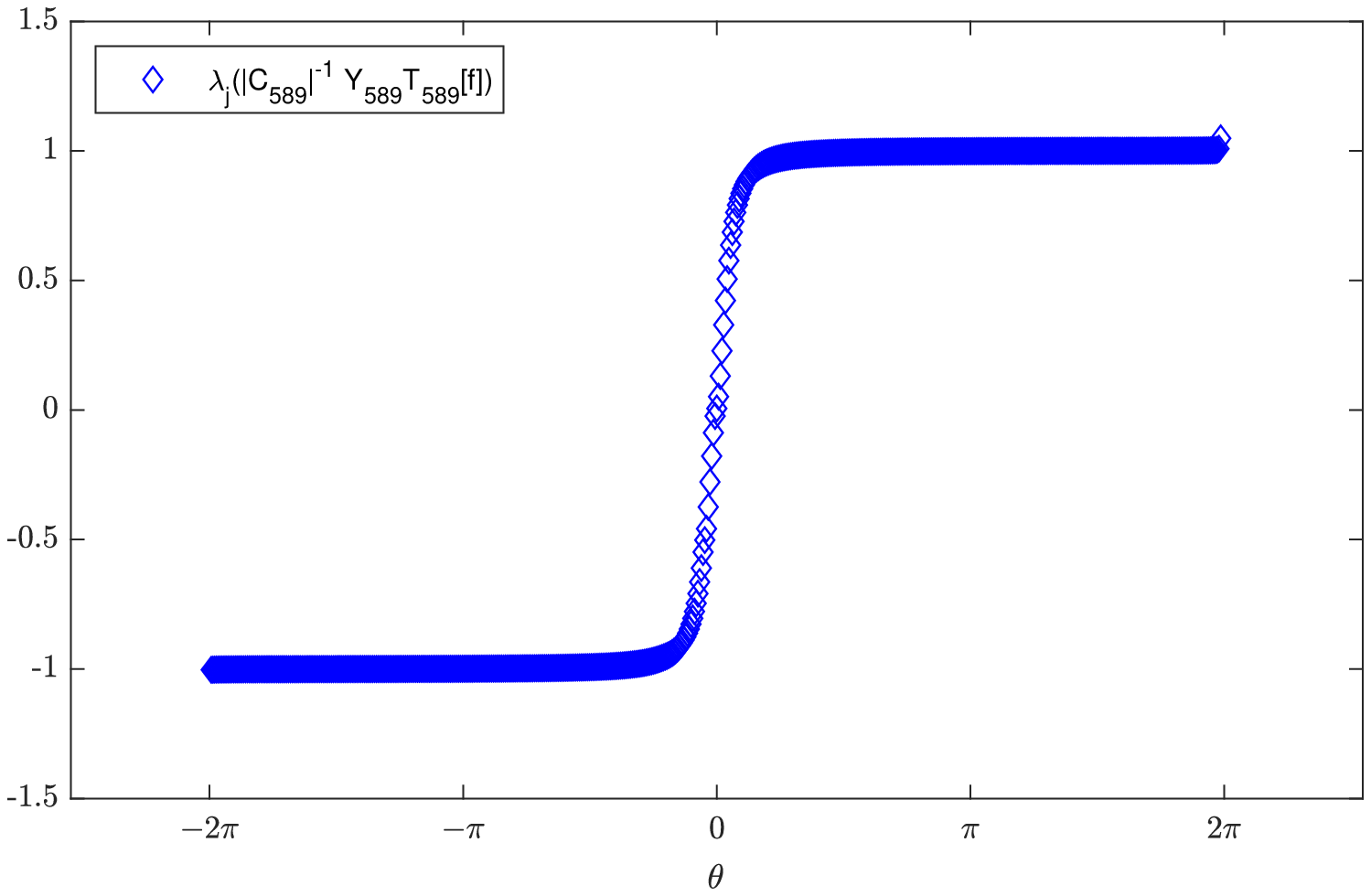}
			\includegraphics[width=0.45\textwidth]{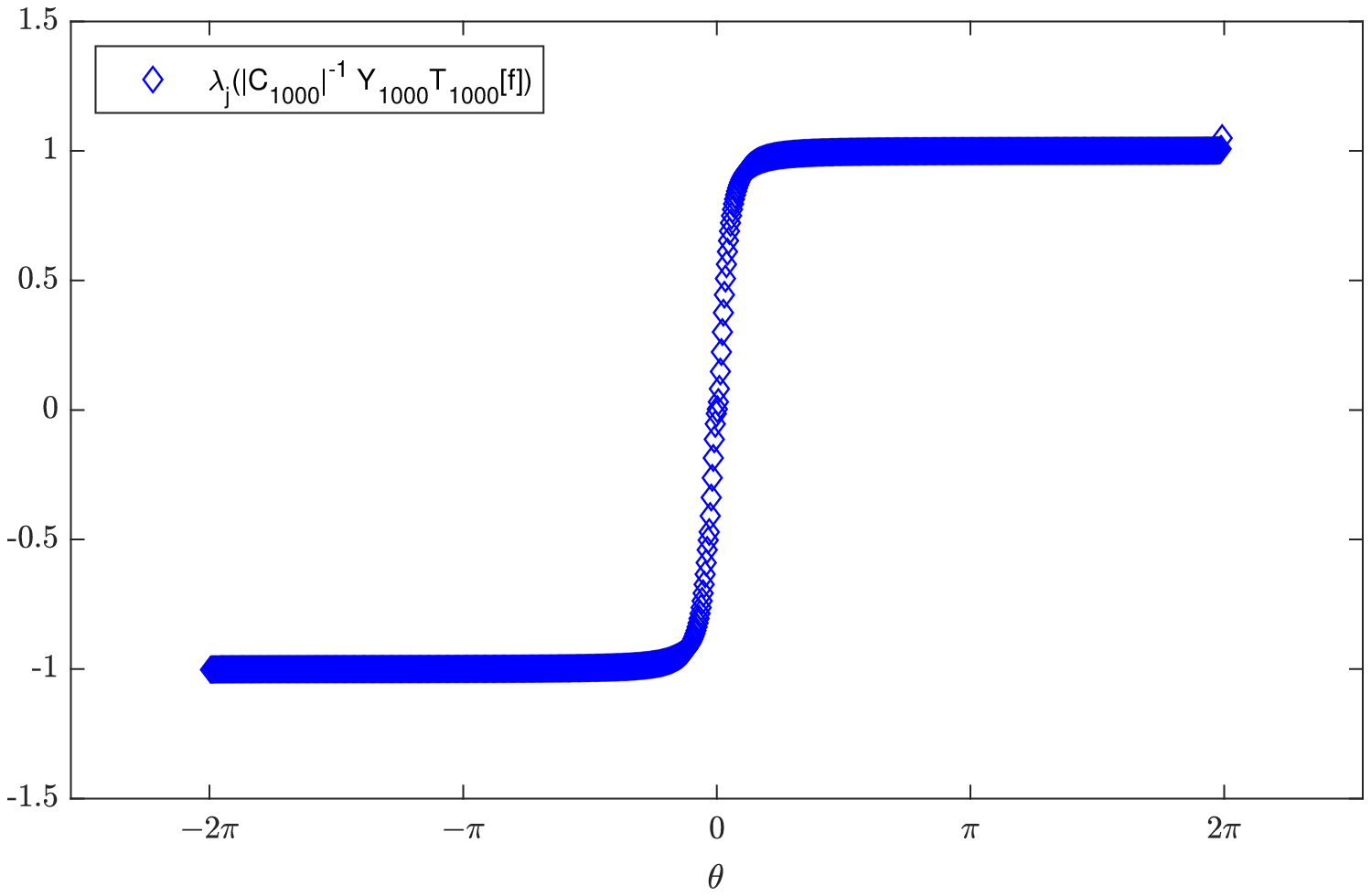}
			\vspace{-10pt}
			\caption{Example~\ref{ex_prec_theta2}, the eigenvalues of $|C_n|^{-1} Y_nT_n[f]$, where  $f(\theta) = \theta^2$, $C_n$ is the Frobenius optimal preconditioner, and $n = 157, 200, 589, 1000$. The greatest eigenvalue $\lambda_n(|C_n|^{-1} Y_nT_n[f])$ is not plotted.}
			\label{fig:prec_fro_theta2}
		\end{figure}
	\end{example}
	
	\begin{example}\label{ex_prec_pwc}
		In this last example, we consider the discontinuous function
		\[
		f(\theta)=\left\{
		\begin{array}{ll}
		5, & \theta\in [-\pi,-\pi/2), \\
		2, & \theta\in [-\pi/2,\pi/2), \\ 
		5, & \theta\in [\pi/2,\pi]. \\
		\end{array}
		\right. 
		\]
		
		In this case, instead of using the argument $\mathbf{B}$,  in Figure~\ref{fig:ex_toeplitz} we show graphically that the property 
		\[
		\{C_n^{+} T_n[f]\}_n \sim_{\sigma}  1,
		\]
		is true for the Strang preconditioner. 
		
		In Figure~\ref{fig:ex_prec_pwc}, we plot the eigenvalues $\lambda_j(|C_n|^{-1} Y_nT_n[f])$, $j = 1,\dots,n-1$, for $n = 500, 1000$. In both cases, the eigenvalue $\lambda_n(|C_n|^{-1} Y_nT_n[f])$ is an outlier of large magnitude and, therefore, we do not plot it in order to make the figures more readable.
		
		The clustering of the spectrum around $\pm1$ numerically confirms the distribution result on the preconditioned matrix sequence $\{|C_n|^{-1} Y_nT_n[f]\}_n$ in a more general hypothesis of Theorem \ref{thm:main_hon-bis-corollary2}.
		\begin{figure}
			\centering
			\includegraphics[width=0.45\textwidth]{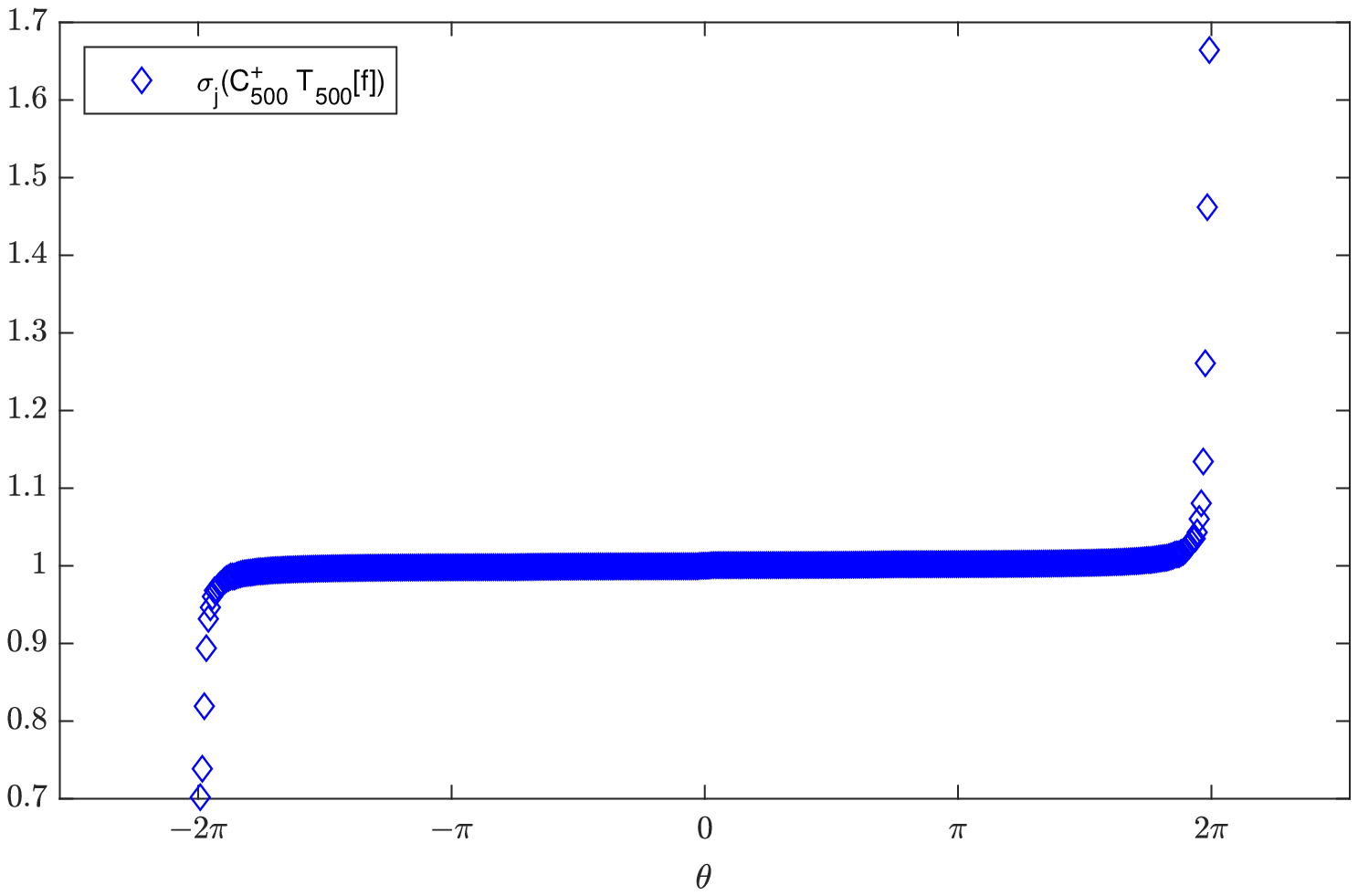}
			\includegraphics[width=0.45\textwidth]{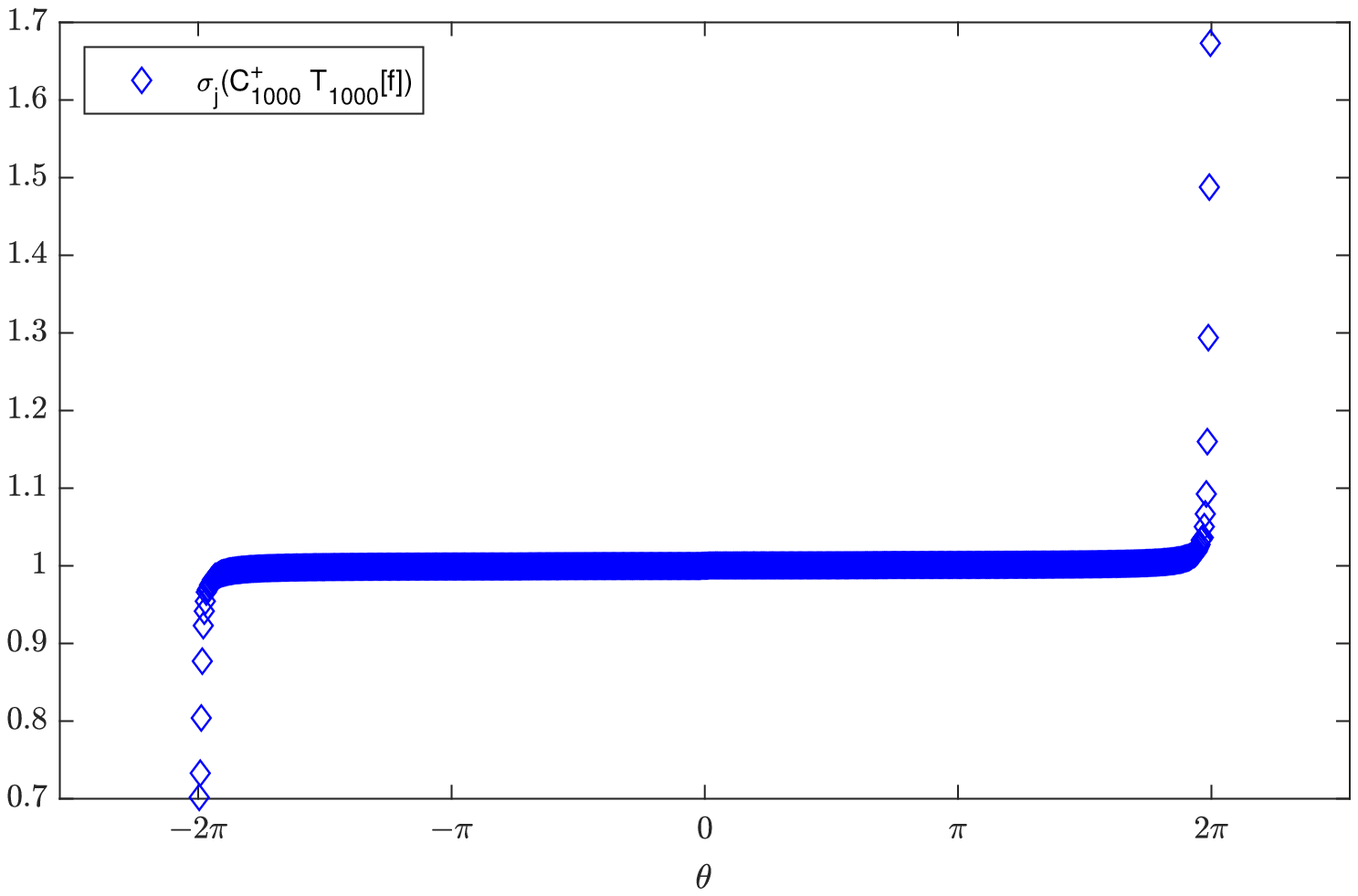}
			\caption{Example~\ref{ex_prec_pwc}, the singular values of $C_n^{+} T_n[f]$, where $f$ is piecewise constant, $C_n$ is the Strang preconditioner, and $n = 500, 1000$.}
			\label{fig:ex_toeplitz}
		\end{figure}
		\begin{figure}
			\centering
			\includegraphics[width=0.45\textwidth]{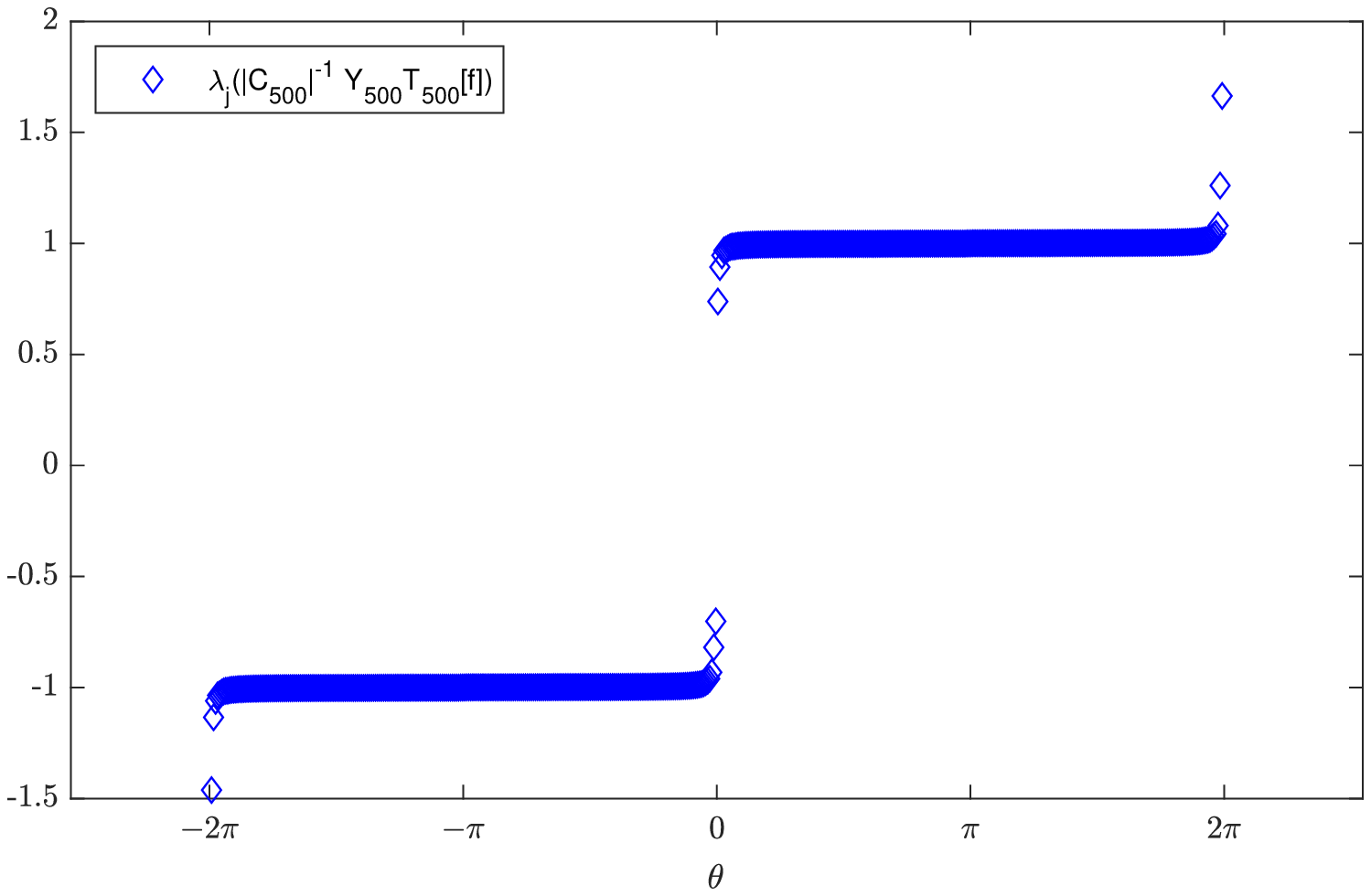}
			\includegraphics[width=0.45\textwidth]{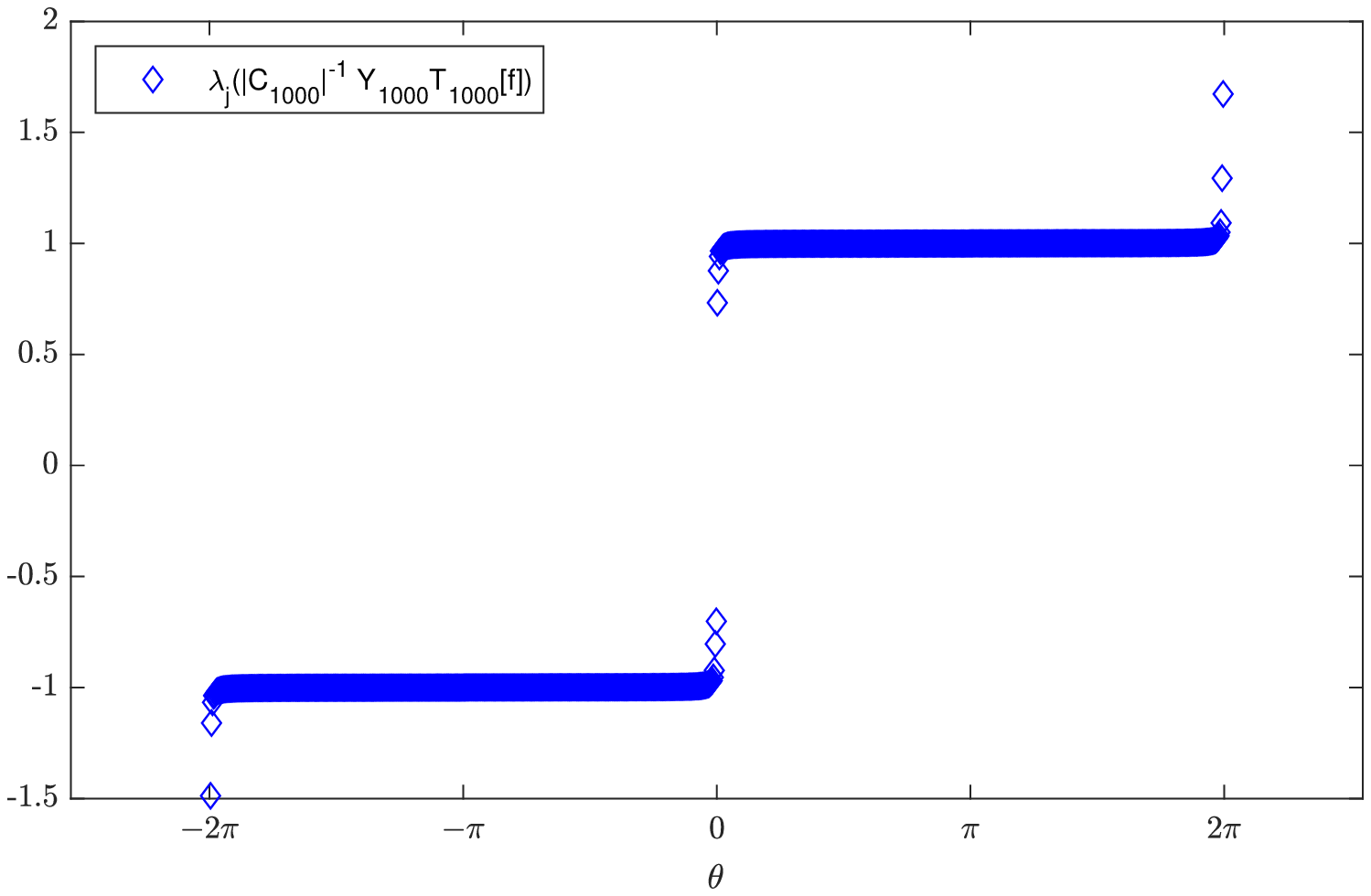}
			\caption{Example~\ref{ex_prec_pwc}, the eigenvalues of $|C_n|^{-1} Y_nT_n[f]$, where $f$ is piecewise constant, $C_n$ is the Strang preconditioner, and $n = 500, 1000$. The greatest eigenvalue $\lambda_n(|C_n|^{-1} Y_nT_n[f])$ is not plotted.}
			\label{fig:ex_prec_pwc}
		\end{figure}
	\end{example}
	
	\section{Conclusions}
	
	We have provided our main theorem that describes the singular and spectral distribution of certain special $2$-by-$2$ block matrix sequences. Included as a special case of the theorem, the symmetric matrix sequence $\{Y_nT_n[f]\}_n$ is essentially distributed as $\pm |f|$. As a consequence, the preconditioned matrix sequence $\{|C_n|^{-1} Y_nT_n[f]\}_n$ is distributed as $\pm1$ provided that a suitable circulant preconditioner $C_n$ is used. A series of numerical examples concerning different generating functions and circulant preconditioners have also been provided to support our theoretical results.
We acknowledge that similar results are given in \cite{mazza-pestana} by using different techniques: while our approach is based on the notion of approximating class of sequences, the derivations in \cite{mazza-pestana}  are obtained by using the powerful *-algebra structure of the GLT sequences.

%\bibliographystyle{siamplain}
%\bibliography{references}

	\section*{References}
	
	\bibliographystyle{siamplain}

\end{document}